\documentclass[a4paper,11pt,superscriptaddress]{article}
\usepackage[top=30mm,left=30mm,bottom=30mm,right=30mm]{geometry}
\usepackage{amsmath}
\usepackage{amssymb}
\usepackage{amsthm}
\usepackage{graphicx}
\usepackage{hyperref}
\usepackage{amsfonts}
\usepackage{bbm}
\usepackage{comment}
\usepackage{color}
\usepackage{tikz-cd}

\usepackage{hyperref}
\definecolor{forestgreen}{RGB}{0, 153, 0}
\hypersetup{
    colorlinks,
    citecolor=forestgreen,
    filecolor=black,
    linkcolor=red,
    urlcolor=blue
}
\hypersetup{linktocpage}

\newcommand{\mc}[1]{\mathcal{#1}}
\newcommand{\mbb}[1]{\mathbb{#1}}
\newcommand{\mf}[1]{\mathfrak{#1}}
\newcommand{\norm}[1]{\left\lVert#1\right\rVert}
\DeclareMathOperator{\sgn}{\text{sgn} }

\newtheorem{theorem}{Theorem}[section]
\newtheorem{corollary}{Corollary}[section]
\newtheorem{lemma}{Lemma}[section]
\newtheorem{proposition}{Proposition}[section]
\theoremstyle{definition}
\newtheorem{definition}{Definition}[section]
\newtheorem{example}{Example}[section]

\newcommand\blfootnote[1]{%
  \begingroup
  \renewcommand\thefootnote{}\footnote{#1}%
  \addtocounter{footnote}{-1}%
  \endgroup
}

\begin{document}
\title{Singular Contact Geometry and Beltrami Fields in Cholesteric Liquid Crystals}
\author{Joseph Pollard\footnote{Mathematics Institute, Zeeman Building, University of Warwick, Coventry, CV4 7AL, United Kingdom.} \and Gareth P. Alexander\footnote{Department of Physics and Centre for Complexity Science, University of Warwick, Coventry, CV4 7AL, United Kingdom.}}


\maketitle

\begin{abstract}
The description of point defects in chiral liquid crystals via topological methods requires the introduction of singular contact structures, a generalisation of regular contact structures where the plane field may have singularities at isolated points. We characterise the class of singularities that may arise in such structures, as well as the subclass of singularities that can occur in a Beltrami field. We discuss questions of global existence, and prove that all singular contact structures with nonremovable singularities are overtwisted. To connect the theory to experiment we also discuss normal and tangential boundary conditions for singular contact structures, and show we can realise all desired boundary conditions except for normal anchoring on a sphere, where a theorem of Eliashberg and Thurston provides an obstruction to having a singular contact structure in the interior. By introducing a singular version of the Lutz twist we show that all contact structures are homotopic within the larger class of singular contact structures. We give applications of our results to the description of topological defects in chiral liquid crystals. 
\end{abstract}

\blfootnote{This work is supported by the UK ESPRC through Grant No. EP/L015374/1.}

\section{Introduction}
Contact topology has long been understood as an important tool in mathematical physics. As the odd-dimensional analogue of symplectic geometry, it is intimately related to the Hamiltonian formalism of classical mechanics~\cite{arnold89}. It is the natural setting for thermodynamics, as well as various problems in the theory of waves and geometric optics~\cite{arnold90}. Combined with methods from geometry and analysis, it is also central to modern developments in topological fluid dynamics~\cite{arnold98}. The fundamental solutions of the incompressible Euler equations are Beltrami fields, vector fields parallel to their curl, and a theorem of Etnyre \& Ghrist~\cite{etnyre00} shows that Beltrami fields can be identified with 1-forms defining contact structures, and vice-versa. Beltrami fields have been extensively studied in their own right, stemming from their significance as force-free magnetic fields~\cite{chandrasekhar57,chandrasekhar58,woltjer58}, their potential to display Lagrangian turbulence~\cite{ginzburg94}, and their relationship with solutions of the Navier--Stokes equations~\cite{constantin88}.

The study of Beltrami fields via contact topology has largely avoided a discussion of singularities -- stagnation points in the flow -- and the corresponding notion of a `singular contact structure' has only been studied recently, and in a different form than we introduce here~\cite{cardona19,miranda18}. 
Nonetheless, singularities do occur in all sorts of physical systems, frequently with central significance to their properties and observed phenomena, so that for physical applications there is a need to include singularities into contact topology. In this paper we define, and investigate, singular contact structures by analogy with the established concept of a singular foliation and motivated by providing a mathematical framework for the description of certain defects in cholesteric liquid crystals and other chiral phases in condensed matter physics, such as skyrmion states in chiral ferromagnets~\cite{milde13}. The connection between contact structures and foliations is now classical, and is explored in detail in the work of Eliashberg \& Thurston on confoliations~\cite{eliashberg97}; our approach is close in spirit to theirs. We restrict ourselves to considering only point singularities of orientable contact structures on 3-manifolds and defer extension to other kinds of singularities to further work.

While the constructions we develop are applicable to the general study of Beltrami fields and contact geometry, we are especially motivated by applications to liquid crystals and the description of their defects. Liquid crystals are optically anisotropic materials that can be described mathematically by a unit line field $N$, called the director, giving at each point in space the local optical axis. In the most common physical theory, the configuration of the director field is given by minimisation of the Oseen--Frank free energy~\cite{ball2017,deGennesProst},
\begin{equation}
  F = \frac{1}{2}\int_M \Bigl( K_1 \bigl( \nabla \cdot N \bigr)^2 + K_2 \bigl( N \cdot \text{curl }N - q_0 \bigr)^2 + K_3 \norm{\nabla_N N}^2 \Bigr) \mu,
  \label{eq:Frank}
\end{equation}
where $M$ is the material domain, a subset of Euclidean space, and $\mu$ is the volume form. The Oseen--Frank free energy consists of three basic types of elastic distortion of the director: splay, $\nabla \cdot N$; twist, $N \cdot \text{curl }N$; and bend, $\nabla_N N$. The $K_i$ are material-dependent elastic constants and $q_0$ is a material-dependent parameter expressing an energetic preference for the director to undergo twist distortion, as well as imparting a sense of handedness, or chirality -- if $q_0 > 0$ the material is left-handed, if $q_0 < 0$ the material is right-handed. A material where $q_0 \neq 0$ is referred to as a chiral, or cholesteric, liquid crystal.

Defects in liquid crystals are points (or lines) where the optical axis is undefined and the director field is singular. For example, the radial configuration $N = x/\norm{x}$ contains a point defect at the origin. Such defects in ordered phases of matter have been studied extensively in both physical and mathematical theories, with methods of topology applied to give a classification of the different defects that can occur. The methods come primarily from homotopy theory~\cite{alexander12,janich87,mermin79,monastyrsky86}, although often the physics requires a finer treatment that also captures some relevant geometric or energetic aspect. This happens, in particular, in phases where the ground state has non-zero spatial gradients and is not simply homogeneous~\cite{mermin79}. An example arises in smectic liquid crystals, where the density is spatially inhomogeneous and results from the theory of measured foliations~\cite{chen09,machon19,poenaru81} have been needed to properly account for the experimentally observed defects. Similarly, in cholesterics the classical homotopy classification of defects does not fully describe the physics of these materials, and geometric methods are required to properly describe the observed defects.

The results we develop here are inspired by recent experiments~\cite{posnjak17} on cholesteric liquid crystals ($q_0 \ne 0$) confined to spherical droplets, and the need to classifiy the chiral point defects that are observed. They are based in the methods of contact geometry. To convey informally the relevance of contact geometry for cholesterics we can observe that, in the absence of defects and boundary conditions, the minimiser (ground state) of~\eqref{eq:Frank} is a unit length Beltrami field, namely
\begin{equation}
  N = \sin(q_0 z) \,\partial_x + \cos(q_0z) \,\partial_y,
\end{equation}
or any vector equivalent to it. More generally, any director $N$ satisfying $N \cdot \textrm{curl } N \approx q_0$ (or even the weaker condition $N \cdot \textrm{curl } N \neq 0$) has a small twist energy and is therefore favoured in the Oseen--Frank free energy Eq. (\ref{eq:Frank}); such a director has a dual 1-form defining a contact structure. Initial steps towards adopting contact topology and geometry in the study of cholesteric liquid crystals have been made recently: in the foundational study~\cite{machon17}, it is explained how the Gray stability theorem for contact structures strongly restricts the dynamics of cholesteric materials and leads to the formation of experimentally-observed structures; in~\cite{pollard19}, the influence of contact topology on the formation of structures in spherical droplets of cholesteric liquid crystal is discussed and it is argued physically that the contact condition underlies the exotic structures observed in experiments~\cite{posnjak17}.

However, despite these initial steps, the application of contact geometry to cholesterics meets with two difficulties. First, if we are to describe the defects that occur in liquid crystals then we need to broaden the definition of contact structures to allow for singularities. The connection between contact structures and Beltrami fields allows an intuitive picture of these singularities to be sketched informally~\cite{pollard19}. If $N$ is a Beltrami field then evaluating the Beltrami equation at any zero gives $\bigl. \textrm{curl } N \bigr|_0 = 0$. Consequently the first term in a Taylor series about the zero must be the gradient of a function, $N = \nabla \phi + \cdots$, and since Beltrami fields are divergence-free the function is harmonic. The higher order terms in the Taylor series are needed to capture the chirality of the singularity and are provided by a solution of $\textrm{curl } V = q_0 \nabla \phi$. It is not sufficient to consider only generic singularities, as a variety of defects with non-generic local structures have been discovered in experiments~\cite{posnjak17}.

Second, most experiments have the liquid crystal confined in a domain with an imposed boundary condition -- a preferred anchoring for the liquid crystal molecules -- and there is a fundamental obstruction to the existence of contact structures compatible with certain boundary conditions. This is the Reeb stability theorem for confoliations of Eliashberg \& Thurston~\cite{eliashberg97}, which we will discuss in \S\ref{sec:global_properties}. The physical consequence of this is the necessary existence of regions of the material with the wrong handedness~\cite{pollard19,ackerman2016} -- i.e., the sign of the twist $N \cdot \textrm{curl } N$ -- in order to satisfy both the boundary conditions and topology.


This paper proceeds as follows. In \S\ref{sec:singularity_theory} we give a brief overview of the tools from singularity theory that we will need. In \S\ref{sec:singular_contact_structures} we introduce singular contact structures and discuss an extension of the Etnyre--Ghrist theorem to singular contact structures and singular Beltrami fields. The result of Etnyre \& Ghrist does not always extend over the singular points, leading to further distinction between the types of singularity where it does extend, those that may occur in a Beltrami fields, and the types where it does not extend, those that may occur in a contact structure (or equivalently, a vector field satisfying $N \cdot \textrm{curl } N \neq 0$ away from the singularities) but not a Beltrami field. In \S\ref{sec:local_properties} we describe the local structure around the singular points, and explain the difference between singularities that are \textit{chiral}, occuring in singular contact structures, and those that are \textit{Beltrami}, occuring in singular Beltrami fields; we give examples of singularities that belong to the former class but not the latter. Our method is constructive and we give local models that correspond to observed singularities in cholesteric liquid crystals. We also give a refined construction for Beltrami fields.

In \S\ref{sec:global_properties} we turn our attention to global properties. We show that singular contact structures can be constructed on any 3-manifold, which follows from the Lutz--Martinet theorem, and prove that they are overtwisted whenever they have any singularities with nonzero index. For manifolds with boundary we consider the two situations where either the contact planes or the Reeb field are tangent to the boundary. For fixed boundary conditions we give necessary and sufficient conditions to extend a singular contact structure over the interior. We discuss a surgery on contact manifolds called the Lutz twist and give an explicit realisation for it as a homotopy through singular contact structures. This establishes that all contact structures are homotopic in the larger class of singular contact structures. Physically this result shows that there is no topological obstruction to chiral liquid crystals attaining their ground state, only an energetic barrier. We conclude with some remarks about a singular version of the Weinstein conjecture for contact structures.

\section{Singularity Theory} \label{sec:singularity_theory}
In this section we review some notions from singularity theory that we will make use of~\cite{arnold88}. A \textit{germ} of a map at a point $p$ in a manifold $M$ is an equivalence class of maps, where two maps $f, h$ are considered equivalent if there exists an open neighbourhood $U$ of $p$ such that $f|_U = h|_U$. This notion extends to germs of functions, vector fields, differentials forms, etc. By an abuse of notation, we will confuse the equivalence class with one of its representatives.

We consider germs of maps from $\mbb{R}^n$ to $\mbb{R}^m$ that have an isolated zero at the origin. Two germs of maps $X, Y$ are considered equivalent, which we denote $X \sim Y$, if there exist germs of diffeomorphisms $f, h$ that map the origin to itself and are such that $X = f \circ Y \circ h^{-1}$. A \textit{singularity} is an equivalence class of germs of maps $X : \mbb{R}^n \to \mbb{R}^m$ with isolated zero at the origin. If $X$ is a vector field on an $n$-manifold $M$ with an isolated zero at a point $p$, we may take a small open neighbourhood $U$ of $p$ and identify $X|_U$ with some singularity. From now on, we only consider the case when $n=3$.

Denote by $j^k X$ the $k$-jet of a map germ, the Taylor series of $X$ truncated at order $k$. A map germ is \textit{finitely determined} if it is equivalent to its $k$-jet for some finite $k$. This turns out to imply that the map germ is real analytic~\cite{arnold88} and it is therefore sufficient to restrict our attention to analytic maps, whose components (in some coordinate system) are polynomials.

The components of a germ of a polynomial vector field (or 1-form) $X$ generate an ideal $I_X$ in the ring $\mbb{R}[x,y,z]$ of polynomials in three variables. The maximum ideal $\mf{m}$ consists of those polynomials that vanish at 0, i.e. those with no constant part. We will say another germ of a vector field $Y$ belongs to the ideal generated by $X$, and write $Y \in I_X$, if each component of $Y$ belongs to the ideal.
\begin{proposition}
  If $X$ is a polynomial vector field and $Y \in \mf{m}I_X$, then $X \sim X+Y$.
\end{proposition}
\noindent See~\cite{arnold88} for a proof. Informally, vector fields belonging to the ideal $\mf{m}I_X$ have monomial terms that are at least `one order higher' than the highest order terms in $X$, and adding them to $X$ does not change the local structure of the singularity.

The \textit{local algebra} of the singularity, is the algebra $Q_X = \mbb{R}[x,y,z]/ I_X$. It is finitely generated when the germ $X$ is finitely determined, and thus has a finite basis consisting of monomials. A generic perturbation of $X$ by adding terms from the local algebra will break the germ apart into germs of singularities of lower complexity; this process is called an \textit{unfolding}.

In \S\ref{sec:local_properties} we establish (Proposition \ref{prop:only_gradients}) that we need only study singularities of the form $\nabla \phi$, for $\phi : \mbb{R}^3 \to \mbb{R}$ a polynomial function. Arnold has given a full classification of singularities arising as gradient fields of functions; an extensive tabulation appears in~\cite{arnold88}. Two gradient vector fields $\nabla \phi, \nabla \psi$ are equivalent if there exist germs of diffeomorphisms $h : \mbb{R}^3 \to \mbb{R}^3$ and $f : \mbb{R} \to \mbb{R}$ such that $\psi = f \circ \phi \circ h^{-1}$.

\begin{example}
  A germ of a singularity $\phi : \mbb{R}^3 \to \mbb{R}$ whose matrix of second derivatives (Hessian) is nondegenerate at the singular point is called \textit{Morse}. By the Morse lemma, these singularities are equivalent to a function germ of the form $\phi = \pm x^2 \pm y^2 \pm z^2$. The number of minus signs is the \textit{Morse index}. A generic perturbation of a non-Morse singularity will break it apart into Morse singularities.
\end{example}

For any vector field with a singularity at the point $p$, we define an index as follows. Choose a small embedded sphere $S$ around $p$. The vector field $X$ determines a map from $S$ to the unit $n$-sphere, which induces a homomorphism from the group $H_2(S^2, \mbb{Z}) \cong \mbb{Z}$ to itself. Consequently the action of this map is of the form $x \mapsto \lambda x$ for some integer $\lambda$ called the \textit{degree} of the map. We define the \textit{index} of $X$ at $p$ to be equal to $\lambda$.
\begin{example}
  If $\phi$ is a Morse singularity, the index is $+1$ if the Morse index is even and $-1$ if it is odd.
\end{example}
\noindent The index couples to the topology of $M$ via the well-known Poincar\'e--Hopf theorem. On a closed manifold $M$, this states that the sum of the indices of any vector field with isolated singularities is equal to the Euler characteristic of $M$; this also holds on compact manifolds with boundary as long as the vector field is transverse to the boundary.

\section{Singular Contact Structures} \label{sec:singular_contact_structures}
A plane field $\xi$ on an oriented 3-manifold $M$ is a rank 2 subbundle of the tangent bundle. If $M$ and $\xi$ are both oriented, then $\xi$ is equal to the kernel of a non-vanishing 1-form $\eta$, and the line bundle $TM / \xi$ is directed by a nowhere-vanishing vector field. Conversely, given a nowhere-vanishing vector field $N$ on a Riemannian manifold $(M,g)$, we have an oriented plane field $\xi$ consisting of the planes orthogonal to $N$, which is defined by the 1-form $\eta = \iota_N g$. We will only consider orientable plane fields in this article.

The plane field is integrable when $\eta \wedge d\eta = 0$, and is then the tangent space of a collection of immersed submanifolds. Collectively these submanifolds form a \textit{foliation}, and the individual submanifolds are \textit{leaves} of the foliation. If $\eta \wedge d\eta \neq 0$ then the plane field is a \textit{contact structure}, and $\eta$ is called a \textit{contact form}. A contact form is \textit{positive} (left-handed) if $\eta \wedge d\eta > 0$ relative to the orientation on $M$, and \textit{negative} (right-handed) if $\eta \wedge d\eta < 0$. Interpolating between foliations and contact structure are the \textit{confoliations}, plane fields defined by 1-forms with $\eta \wedge d\eta \ge 0$ or $\eta \wedge d\eta \le 0$, depending on the handedness. Confoliations were introduced by Eliashberg \& Thurston in~\cite{eliashberg97} to study the perturbation of foliations into contact structures.

The \textit{Reeb field} of a contact form $\eta$ is the vector field $R$ defined by $\iota_R \eta =1$ and $\iota_R d\eta = 0$. The connection between Reeb fields and Beltrami fields is given by the following theorem of Etnyre \& Ghrist.
\begin{theorem} (\cite[Theorem 2.1]{etnyre00})\label{thm:etnyre-ghrist}
  Let $(M, g)$ be a Riemannian 3-manifold. Any smooth, nonsingular, Beltrami field is the Reeb field of some contact form $\eta$ on $M$, and conversely the Reeb vector field of a contact form on $M$ is a smooth, nonsingular, Beltrami field with respect to some metric on $M$.
\end{theorem}
\noindent This theorem follows from the existence of special metrics associated to the contact form $\eta$. A Riemannian metric $g$ is called \textit{compatible} with $\eta$ if $R$ is the unit normal to the planes of $\xi = \text{ker } \eta$ in this metric. Such metrics are locally parameterised by maps $J : TM \to TM$ satisfying $J^2 = -I + \eta \otimes R$, which behave as almost complex structures on the planes of $\xi$. Given such a map, the associated compatible metric $g$ is defined by
\begin{equation}
  g(X, Y) = d\eta(X, JY) + \eta(X) \eta(Y),
\end{equation}
for any vector fields $X, Y$~\cite{blair10}. Alternatively, the existence of compatible metrics can be deduced from a slightly more general result, a proof of which can be found in~\cite{farber03}. We denote the star operator of a Riemannian metric $g$ by $\star_g$.
\begin{lemma} \label{lemma:metric_construction}
  For any 1-form $\alpha$ and 2-form $\beta$ such that $\alpha \wedge \beta > 0$, it is possible to find a Riemannian metric $g$ on $M$ for which $\star_g \alpha = \beta$.
\end{lemma}
\noindent For a contact form $\eta$, we may apply this lemma to $\alpha = \eta, \beta = d\eta$ to construct the desired compatible metric $g$, and we will have $\star_g \eta = d\eta$. It follows that the Reeb field satisfies $\text{curl }R = R$ with respect to this metric, and is therefore Beltrami. One can also consider the broader class of \textit{weakly compatible metrics}, those for which the Reeb field is orthogonal to the contact planes but we do not control its magnitude~\cite{etnyre12}. These are locally determined by the almost complex structure $J$ as well as positive functions $\lambda, \rho$, where $\rho = \norm{R}$ and $\star d\eta = \lambda \eta$.

Now we turn to singular contact structures. Motivated by experiments on cholesteric liquid crystal droplets~\cite{posnjak17,pollard19}, we focus on vector fields $N$ with isolated zeros, satisfying the condition $g(N, \textrm{curl } N) \neq 0$ away from the set of singular points. Normalising such vector fields on the complement of their singular points yields the director of a cholesteric liquid crystal with point singularities. The space of vector fields orthogonal to $N$ at each point $p$ determines a plane in the tangent space $T_p M$, except at singular points where all tangent vectors are orthogonal to $N$. The natural notion of a `singular plane field' corresponding to this idea is a collection of planes $\xi_p \in T_p M$ depending smoothly (or analytically) on $p \in M$ that degenerates at a collection of isolated points.
\begin{definition} \label{def:singular_plane_fields}
  A 1-form $\eta$ defines a \textit{singular plane field} if there is a finite set $\Sigma \subset M$ of points $p$ for which $\eta_p = 0$, and $\eta_q \neq 0$ for $q \notin \Sigma$. The singular plane field is the kernel of $\eta$. On the complement of $\Sigma$ it is a codimension 1 subbundle of $T(M - \Sigma)$, degenerating to a codimension 0 bundle on $\Sigma$. A 1-form $\eta$ is a \textit{singular contact form} if and only if for all $q \notin \Sigma$, $(\eta \wedge d\eta)_q \neq 0$. The singular plane field defined by a singular contact form is a \textit{singular contact structure}. The singular contact structure is positive (or left-handed) if $\eta \wedge d\eta \geq 0$ with respect to the orientation on $M$, and negative (right-handed) if $\eta \wedge d\eta \leq 0$. A singular plane field is a \textit{singular foliation} if it is defined by $\eta$ with $\eta \wedge d\eta = 0$.
\end{definition}
\noindent To avoid confusion, when we use the term `plane field' we will always mean in the familiar sense, and when we use the term `singular plane field' we will always mean that there is at least one singularity. The same applies to the terms `singular contact structure' and `singular Beltrami field'.

Although we do not have a unique Reeb field in the singular case, on the complement of the singular points the kernel of $d\eta$ is still 1-dimensional, and we can make the following definition.
\begin{definition}
  A vector field $R$ is \textit{Reeb-like} for a singular contact form $\eta$ if $\iota_R d\eta = 0$.
\end{definition}
\noindent Note that the singular points of Reeb-like fields always agree with the singularities of the contact form.

We would like to give a singular version of the Etnyre--Ghirst correspondence, establishing a duality beteen singular contact structures and singular Beltrami fields. One direction is simple.
\begin{proposition}
  Let $(M, g)$ be a Riemannian manifold and $X$ a singular Beltrami field, $X = \text{curl }X$, with singular set $\Sigma$. Then $\eta = \iota_X g$ is a singular contact form with singular set $\Sigma$, and $\eta = \star_g d\eta$. Moreover, $X$ is Reeb-like for $\eta$.
\end{proposition}
\begin{proof}
  Let $\mu$ be the volume form of $g$. For $\eta = \iota_X g$, the relation $X = \text{curl }X$ becomes $\iota_X g = \star_g d\iota_X g$, which implies that $\eta = \star_g d\eta$. Then $\eta \wedge d\eta = \norm{\eta}_g^2 \mu$, and since $\norm{\eta}_g$ is non-negative and vanishes only when $\eta$ does, it follows that $\eta$ is a singular contact form. Also, $\eta$ vanishes only when $X$ does, so they share the same singular set.

  For the second claim, we compute that $\iota_X d\eta = \iota_X \star_g \iota_X g = 0$, so that $X$ is Reeb-like for $\eta$.
\end{proof}
\noindent For the converse statement to hold, given a singular contact form $\eta$ we need to find a Riemannian metric $g$ for which $\eta = \star_g d\eta$, so that the Reeb-like fields of $\eta$ will become Beltrami (though perhaps not divergence-free) with respect to the metric $g$. The Etnyre--Ghrist correspondence is local and therefore continues to hold away from the singularities, ensuring that such a metric will exist on the complement of the zeros of $\eta$, but there is no reason to assume it will extend over the singular points. In fact, in the next situation we will give an example of a germ of a singular contact form for which the metric does not extend over the singularities. To study this phenomenon, we first introduce a singular version of Lemma \ref{lemma:metric_construction}.

\begin{lemma} \label{lemma:singular_metric_construction}
  Let $\alpha$ be a 1-form and $\beta$ a 2-form on a compact 3-manifold $M$, with mutual zero set $\Sigma$ consisting of finitely many isolated points. Suppose that $(\alpha \wedge \beta)_p \ge 0$ for each $p \in M$, vanishing if and only if $p \in \Sigma$. There exists a Riemannian metric $g$ defined on $M- \Sigma$ such that $\star_g \alpha|_{M - \Sigma} = \beta|_{M - \Sigma}$, and furthermore the star operator of $g$ extends as a positive semi-definite matrix over $\Sigma$.
\end{lemma}
\begin{proof}
  The existence of the desired Riemannian metric is immediate from Lemma \ref{lemma:metric_construction}, but we will construct it by hand using the method of~\cite{blair10} to show that $\star$ operator can be chosen so that it becomes degenerate rather than undefined on the singular point.

  As $\beta|_{\text{ker }\alpha} \ge 0$, vanishing only on $\Sigma$, the kernel of $\beta$ is 1-dimensional, spanned by some vector field $Y$ vanishing on $\Sigma$ and such that $\alpha(Y) \ge 0$, with equality only at the points of $\Sigma$. We can find a Riemannian metric $h$ such that $\alpha$ is dual to $Y$ in this metric. Locally, we can choose a pair of sections $X_1, X_2$ of $\text{ker }\alpha$ such that $X_1, X_2, Y$ are orthogonal with respect to $h$, and the vector fields $X_1, X_2$ vanish only on $\Sigma$. Define a $2 \times 2$ matrix by $A_{ij} = \beta(X_i, X_j)$. The fact that $\beta|_{\text{ker }\alpha} \ge 0$, vanishing only on $\Sigma$, implies that $A$ also vanishes only on $\Sigma$.

  A nonvanishing skew-symmetric matrix can be uniquely polarised, $A = GF$, where $G$ is positive-definite and symmetric and $F$ is orthogonal. When the skew-symmetric matrix is allowed to be singular, we may still define a (not necessarily unique) polar decomposition by a limit. Let $A_k$ be a sequence of non-vanishing skew-symmetric matrices tending to $A$. Their polar decompositions $A_k = G_k F_k$ exist. As $O(2)$ is compact, there is an $F \in O(2)$ and a subsequence $F_{k_j}$ such that $F_{k_j} \to F$ as $j \to \infty$. Then $A = F\sqrt{A^T A}$.

  Take $G = \sqrt{A^TA}$, a symmetric, positive semi-definite matrix, where positive definiteness fails only on the set where $A=0$, which is $\Sigma$. Therefore, $G$ defines a Riemannian metric on $\text{ker }\alpha$. Extending $G$ so that it agrees with $h$ in the direction $Y$ gives a positive semi-definite matrix $S$. Let $S$ define a star operator $\star$; this map satisfies $\star \alpha = \beta$. By setting $g = |S|^{-1/3}S$, we define a Riemannian metric on $M- \Sigma$ which is such that $\star_g = \star$ on $M-\Sigma$.
\end{proof}

\noindent A special case of this lemma is the existence of singular star operators associated to singular contact forms, and the second part of the Etnyre--Ghrist correspondence for singular contact forms.
\begin{proposition} \label{prop:compatible_metrics}
  Let $\eta$ be a singular contact form on $M$ with singularity set $\Sigma$. There exists a Riemannian metric $g$ defined on $M - \Sigma$ for which $\star_g \eta|_{M - \Sigma} = d\eta|_{M - \Sigma}$, and the vector field $R$ dual to $\eta$ is Reeb-like, orthogonal to the kernel of $\eta$, and parallel to its own curl.
\end{proposition}

\section{Local Properties of Singular Contact Structures} \label{sec:local_properties}
Now we turn to the description of the local structure of the singularities in singular contact structures, characterising the chiral and Beltrami singularities and showing that the latter class is a proper subclass of the former. First we establish that only singularities occuring in foliations can occur in singular contact structures.

Given a germ of an analytic 1-form $\eta$, we define the \textit{sufficient jet} of $\eta$ to be the jet $j^k\eta$, where $k< \infty$ is the smallest integer such that $\eta$ is equivalent (in the sense of \S\ref{sec:singularity_theory}) to its $k$-jet $j^k \eta$; this integer is referred to as the \textit{order} of the sufficient jet. If $\eta$ is a singular contact form, we define the \textit{contact jet} to be the jet $j^k \eta$, with $k$ the smallest integer such that $\eta$ is equivalent to $j^k \eta$ and additionally that $j^k \eta$ is contact.

We will now establish that the order of the contact jet is strictly greater than that of the sufficient jet. In order to do this, we need further notions from singularity theory. The classification of germs of maps may be reduced to the classification of germs whose components are \textit{semiquasihomogeneous}~\cite{arnold88}. A polynomial $\phi$ is called \textit{quasihomogeneous} with degree $d$ and exponents $v_1, v_2, v_3$ if $f(\lambda^{v_1}x_1, \lambda^{v_2}x_2, \lambda^{v_3}x_3) = \lambda^d f(x_1, x_2, x_3)$, and semiquasihomogeneous if it can be written $\phi = \phi_0 + \sum_i c_i \psi_i$ for $\phi_0$ quasihomogeneous with some degree and exponents, $c_i$ constants, and $\psi_i$ a basis for the local algebra $Q_{\nabla \phi}$. A vector field (or 1-form) is called quasihomogeneous with degree $d$ and exponents $v_1, v_2, v_3$ if each of its components with respect to some coordinate basis are quasihomogeneous polynomials with degree $d$ and exponents $v_1, v_2, v_3$, and semiquasihomogeneous if it can be written $X = X_0 + \sum_i c_i Y_i$ with $X_0$ quasihomogeneous and $Y_i \in I_{X_0}$.

\begin{proposition} \label{prop:contact_jet}
  The order of the contact jet of a germ of a singular contact form $\eta$ is strictly greater than the order of the sufficient jet of $\eta$.
\end{proposition}
\begin{proof}
  Evidently the order of the contact jet is at least as large as the order of the sufficient jet. Moreover, the statement is obviously true if $\eta$ is equivalent to $d\phi$ for some $\phi$, so suppose not. Assume for contradiction that the order of the contact jet is equal to the order of the sufficient jet $k$. We can express $j^k \eta = \sum_i \eta_i \xi_i$ in some coordinate basis $x_i$. Since it suffices to prove the statement for any germ equivalent to $\eta$, we may assume $\eta$ has been put into normal form, i.e. is semiquasihomogeneous.

  Suppose first that $\eta$ is quasihomogeneous. Then we may apply~\cite[\S12.2 Proposition 2]{arnold12}, which says we may in fact, after making a coordinate transformation, reduce to the case where each component $\eta_i$ is a homogeneous polynomial of degree $k_i$. By Proposition \ref{prop:compatible_metrics} there is a star operator $\star$ defined on the complement of the singular point such that $\star j^k \eta = d j^k \eta$. Then the condition $\star j^k \eta = d j^k \eta$ implies that the order of the largest monomial term in $\star \eta_i$ is at most $\max(k_j-1, k_l-1)$, where $i,j,l$ are distinct indices. There is some component of $\star \eta$, say the $dx^2 \wedge dx^3$ component, that contains monomials of order at least $k$. Since we are equating this monomial with another monomial of order $\max(k_2-1, k_3-1)$, we conclude that $k < \max(k_2, k_3)$, which implies that $j^k \eta$ contains monomials of order higher than $k$, contradicting the fact that $j^k \eta$ is the sufficient jet.

  If $\eta$ is not quasihomogeneous, then we may write it as $\eta = \eta_0 + \alpha$, where $\alpha \in Q_\eta$ is a polynomial 1-form whose terms consist of monomials of order strictly lower than $k$. These extra terms do not impact the argument we have just given in the quasihomogeneous case, so the conclusion still holds.
 \end{proof}
\noindent In particular, this implies that germs of singular contact forms have terms of at least quadratic order. From Proposition \ref{prop:contact_jet} we deduce our first result on the local structure of the singularities in a singular contact structure.
\begin{proposition} \label{prop:only_gradients}
  Let $\eta$ be a germ of an analytic singular contact form with sufficient jet of order $k$. Then $j^k \eta = d\phi$, for some germ of a function $\phi$.
\end{proposition}
\begin{proof}
  By Proposition \ref{prop:contact_jet} we have $\eta = j^k \eta + \nu$, where $\nu$ is an analytic 1-form containing terms of order strictly larger than $k$. Then $\eta \wedge d\eta = j^k \eta \wedge d(j^k \eta) + \mf{m}^{2k}$. As we are working with germs, the term $j^k \eta \wedge d(j^k \eta)$ dominates and we neglect the higher order terms. We know that $j^k \eta$ is not a singular contact form, so either $j^k \eta \wedge d(j^k \eta)$ vanishes, or it changes sign over some surface containing the singular point. If the latter were true then $\eta$ could not possibly be a singular contact form, so we conclude that $j^k \eta$ defines a singular foliation.

  We must show that this singular foliation is defined by a closed 1-form $d\phi$, for some function $\phi$. Provided the singular set of a foliation has codimension at least 3 -- which is the case here, as the singular set consists of isolated points in a 3-manifold -- then the existence of such a function is assured by Malgrange's Singular Frobenius Theorem \cite{malgrange76}. This tells us there exist germs of analytic functions, $\phi, \psi$ with $\psi(0) \neq 0$ such that $j^k \eta = \psi d\phi$. Since $k$ was the sufficient jet, we must have $\psi = 1$, otherwise $j^k \eta$ would contain terms in $\mf{m}I_{j^k \eta}$, contradicting its sufficiency.
 \end{proof}
\begin{corollary}
  At any point where a singular contact form $\eta$ vanishes, $d\eta$ also vanishes.
\end{corollary}
\noindent It is interesting to observe that the last two propositions are also true in a neighbourhood of any nonsingular point in a singular contact structure. The Darboux theorem implies that in a neighbourhood of such a point there are coordinates such that the contact structure is defined by the 1-form $\eta = dz +xdy$. The sufficient jet of this 1-form is $j^0 \eta = dz$, an exact form; however, we require the 1-jet to make it contact.

Proposition \ref{prop:only_gradients} shows that in order to study germs of singular contact structures it suffices to study germs of foliations defined by closed 1-forms and `perturb' them by adding higher order terms -- all singular contact forms are realised this way. Since we are working with germs it suffices to consider perturbations to linear order in some parameter. It is then natural to adopt some of the ideas of Eliashberg \& Thurston~\cite{eliashberg97} on the perturbation of foliations into contact structures.
\begin{definition}
  Let $\phi$ be a singularity. We say $\phi$ is \textit{chiral} if $d\phi$ can be linearly perturbed into a positive singular contact form, i.e. if there exists a germ of a 1-form $\nu \in \mf{m}I_{d\phi}$ such that the 1-form $\eta = d\phi + t\nu$ satisfies
  \begin{equation}
   \biggl. \frac{\partial (\eta_t \wedge d \eta_t)}{\partial t} \biggr|_{t=0} \ge 0,
  \end{equation}
  with equality only at the zero of $d\phi$.
\end{definition}
\noindent If $d\phi$ can be linearly perturbed into a positive singular contact structure, it can be linearly perturbed into a negative singular contact structure, and vice-versa (i.e., by changing the sign of $\nu$). Therefore, there is no distinction between the types of point defects that can occur in left-handed and right-handed cholesterics. As discussed in the previous section, it is, however, necessary to distinguish between singularities that are chiral, and those that occur in a Beltrami field.
\begin{definition}
  Let $\phi$ be a germ of a singularity. We say $\phi$ is a \textit{Beltrami singularity} if there is a germ of a 1-form $\nu$ such that $\eta = d\phi + \nu$ is a singular contact form, and furthermore that there exists a Riemannian metric $g$ such that $\star_g \eta = d\eta$.
\end{definition}
\noindent Clearly Beltrami singularities are a subset of chiral singularities.

Next we give a lemma describing the structure of the level sets of a singularity $\phi$ of index $k$. The leaves of the foliation of $\mbb{R}^3$ defined by $d\phi$ are the level sets of $\phi$. We call a leaf singular if it contains the singular point. Take a small $\epsilon > 0$ and consider a closed ball $B$ of radius $\epsilon$ centred at the origin in $\mbb{R}^3$. Let $0<\delta \ll \epsilon$ be a regular value of $\phi$, and let $L^\pm = \phi^{-1}(\pm \delta)$ be two nonsingular leaves of the foliation. Define $F^\pm := L^\pm \cap B$ to be the parts of these leaves that lie in the ball $B$. The sets $F^\pm$ are compact 2-manifolds, possibly with boundary. The index $k$ of $\phi$ is related to the Euler characteristic of $F^\pm$ by $\chi(F^\pm) = 1 \pm k$, see~\cite{arnold78}.
\begin{lemma} \label{lemma:level_set_structure}
  If the level sets $L^\pm$ are compact, then $|k| = 1$, $F^{\sgn(k)}$ is diffeomorphic to a sphere, while $F^{-\sgn(k)}$ is empty. If they are not compact, then $F^{-\sgn(k)}$ is diffeomorphic to the compact manifold obtained by removing $|k|+1$ disks from $S^2$, while $F^{\sgn(k)}$ consists of $|k| + 1$ disks.
\end{lemma}
\begin{proof}
  First suppose $k > 0$. Let $L_0 := \phi^{-1}(0)$ be the singular level set. It is a closed and connected subset of $\mbb{R}^3$. Suppose it is compact. Then nearby level sets are all compact. Conversely, if it is not compact then all nearby level sets are not compact. Thus it suffices to consider both $L^\pm$ compact or both noncompact.

  If $L_0$ is not compact then it must be unbounded. It divides $\mbb{R}^3$ into pieces. The `outside' of $L_0$ consists of the level sets with $\phi < 0$, while the `inside' consists of the level sets $\phi > 0$. Each level set on the outside of $L_0$ is also connected. Removing the singular leaf disconnects the set $\phi \leq 0$, and therefore we conclude that the leaf $L^+$ is not connected. If the singular leaf is compact it must be a single point, so the inside level sets are empty.

  Suppose the level sets are compact and that $\delta$ is chosen small enough that $L^+$ does not intersect the boundary of $B(\epsilon)$. Then $F^+ = L^+$ is a compact 2-manifold without boundary, and $\chi(F^+) = 2-2g$, where $g$ is the genus of $F^+$. Since $\chi(F^+) = 1+k$, it follows that $g = (1 - k)/2$. The genus is always a non-negative integer, so we conclude that $k = 1$ and $g=0$, so that $F^+$ is diffeomorphic to a sphere.

  Now suppose $L^\pm$ are not compact. No matter which values of $\delta$ and $\epsilon$ we choose, the sets $L^\pm$ must intersect the sphere. This implies $F^-$ is a connected, compact 2-manifold with boundary, diffeomorphic to a sphere with some number of disks removed, while $F^+$ is a disconnected, compact 2-manifold with boundary. By the classification of 2-manifolds we conclude that $\chi(F^-) = 2-b$, where $b$ is the number of boundary components. It follows that $b = k+1$. $F^+$ is diffeomorphic to the part of the sphere that remains after removing $F^-$, consists of $k+1$ pieces, each diffeomorphic to a closed disk.

  For $k < 0$, the outside is $\phi > 0$ and the inside $\phi < 0$, and the signs in the above argument change accordingly. We have $\chi(F^\pm) = 1 \mp |k|$. The argument for the compact sets is the same, and we find that we must have $k = -1$ and the level sets are again spheres. When the level sets are not compact, $F^+$ is a connected, compact 2-manifold with boundary, diffeomorphic to a sphere with some number of disks removed, while $F^-$ is a disconnected, compact 2-manifold with boundary. Otherwise the same conclusions hold.
 \end{proof}
\noindent Informally, the level sets of a singularity of index $k$ are either spheres, or the singular level set is $|k|+1$ cones adjoined to the singular point, each filled with rounded cones and surrounded by connected surfaces. See Fig. \ref{fig:levelset}.

\begin{figure}[tb]
\centering
\includegraphics[width=1.0\linewidth]{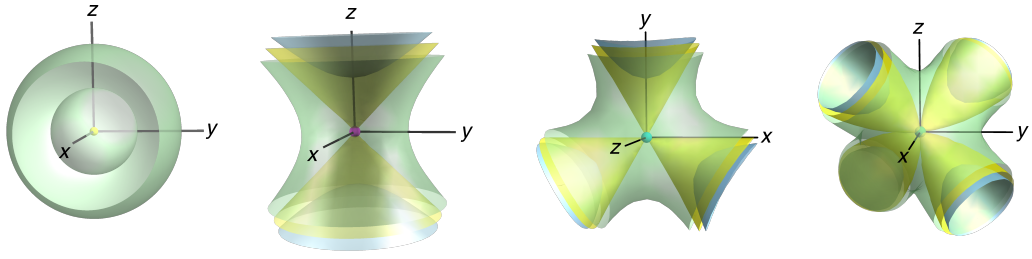}
\caption{Level sets of singularities. From left to right: singularity of Morse index 0 and index $+1$, singularity of Morse index 1 and index $-1$, the singularity $D_4^-$ with index $-2$, and the singularity $T_{4,4,4}$ with index $-3$ (the names of the singularities are those given in Arnold's classification~\cite{arnold88}). All four of these singularities occur in experiments on cholesteric liquid crystal droplets~\cite{posnjak17,pollard19}.}
\label{fig:levelset}
\end{figure}

\subsection{Chiral Singularities}
In this section we prove a theorem characterising chiral singularities. Recall that the Hodge Laplacian on a Riemannian manifold $(M, g)$ is defined by $\Delta_g = d\delta + \delta d$, and that a differential form $\omega$ is harmonic if $\Delta_g \omega = 0$. We say a differential form $\omega$ on $M$ is \textit{intrinsically harmonic} if there exists some metric $g$ on $M$ for which it is harmonic. For closed 1-forms the property of being intrinsically harmonic is equivalent to a certain topological property of the leaves of the foliation it defines.
\begin{theorem} (Calabi's Theorem \cite{calabi69,farber03})
  Let $\omega$ be a closed 1-form with isolated singularities on a closed $n$-manifold. $\omega$ is intrinsically harmonic if and only if for every nonsingular point $p \in M$ there is a closed loop that is transverse to the foliation and passes through $p$.
\end{theorem}
\noindent It is clear that being intrinsically harmonic is a property of the foliation and not of the choice of defining 1-form, and therefore is a property of a germ rather than any particular representative of that germ.

Given a singular star operator $\star$, we can still define a codifferential $\delta = \star d \star$ and a Laplacian $\Delta = d\delta+\delta d$. Given a $k$-form $\omega$ whose singular set is the same as that of the singular set of $\star$, we say that \textit{$\omega$ is harmonic with respect to $\star$} if $\Delta \omega = 0$.
\begin{theorem} \label{thm:chiral_characterisation}
  Let $\phi$ be a germ of a real analytic function $\phi : \mbb{R}^3 \to \mbb{R}$. Then $\phi$ is chiral if and only if there exists a metric $g$ on the complement of the zero, such that the star operator of $g$ extends over the singular point as a symmetric, positive semi-definite matrix, and $d\phi$ is harmonic with respect to $g$ on the complement of the zero.
\end{theorem}
\begin{proof}
  First suppose $\phi$ is harmonic with respect to a metric $g$ defined on the complement of the zero. Then $\star_g d\phi$ is closed and, since we are working with germs, exact: $\star_g d\phi = d \nu$ with $\nu \in \mf{m}I_{d\phi}$. Setting $\eta_t = d\phi + t\nu$ for $t > 0$, we have
  \begin{equation}
    \biggl. \frac{\partial (\eta_t \wedge d \eta_t)}{\partial t} \biggr|_{t=0} = d\phi \wedge d \nu = \norm{d\phi}^2_g \mu > 0,
  \end{equation}
  on the complement of the singular point, where $\mu$ is the volume form and $\norm{\cdot}_g$ the norm of the Riemannian metric $g$. We need to extend $\eta$ over the origin. As the star operator extends over the origin, then $\eta$ extends over the origin by defining it to be zero there, and $\eta_t \wedge d\eta_t = 0$ at the origin. This gives a linear perturbation of $d\phi$ into a germ of a singular contact form, so $\phi$ is chiral.

  Now suppose $\phi$ is chiral, so that there exists a 1-form $\nu$ such that $d\phi \wedge d\nu \ge 0$ with respect to the orientation on $\mbb{R}^3$, vanishing only at the origin. We can apply Lemma \ref{lemma:singular_metric_construction} with $\alpha =d\phi$ and $\beta = d\nu$ to obtain a metric $g$ on the complement of the zero such that $\star_g d\phi = d\nu$. The lemma ensures that the star operator can be extended.
 \end{proof}
Theorem \ref{thm:chiral_characterisation} shows that the property of being chiral is nontrival, and that there do exist achiral function germs.
\begin{corollary} \label{cor:no_spheres}
  Germs of functions with spherical level sets cannot be chiral.
\end{corollary}
\begin{proof}
  Such functions violate the maximum principle, and as such cannot be harmonic, even on the complement of the singular point.
 \end{proof}
\noindent In particular, singularities of Morse index 0 and 3 are achiral. In real cholesteric materials such defects frustrate the energetically preferred chiral twisting, but nonetheless they still occur, where they sit on the boundary between regions of left-handedness and regions of right-handedness~\cite{pollard19,ackerman2016}.

There are other ways of arriving at Corollary \ref{cor:no_spheres} without using Theorem \ref{thm:chiral_characterisation}. Proposition \ref{prop:compatible_metrics} shows the existence of Riemannian metrics on the complement of the singular point for which the singular contact form is coclosed. If the Reeb-like fields were transverse to a sphere, such a metric could not exist. Interestingly there is a theorem of Eliashberg \& Thurston that allows us to draw this conclusion via a purely topological argument.
\begin{theorem} (Reeb Stability Theorem for Confoliations~\cite{eliashberg97}) \label{thm:reeb_stability}
  Let $\xi$ be a confoliation on $S^2 \times \mbb{R}$ which is transverse to the lines $p \times \mbb{R}$ for any $p \in S^2$ and tangent to $S^2 \times 0$. Then there is a neighbourhood $U$ of $S^2 \times 0$ such that $\xi|_U$ is a foliation diffeomorphic to the fibration of $S^2 \times [-1,1]$ by spheres.
\end{theorem}
\noindent This theorem also has consequences for the behaviour of singular contact structures near a boundary. Many experiments are performed on spherical cholesteric droplets. If normal anchoring is imposed on the surface of the droplet, the director is orthogonal to the boundary, and consequently Theorem \ref{thm:reeb_stability} applies: the material cannot assume a single handedness throughout the interior of the droplet. Physically, the material meets the topological requirements of Theorem \ref{thm:reeb_stability} by having regions of handedness reversal localised in a neighbourhood of the droplet boundary, while in the interior it is singular contact. As these regions are energetically disfavoured, there is a strong tendancy to minimise their size, which in turn leads to stabilisiation of exotic structures such as defects described by degenerate singularities~\cite{pollard19}.

\subsection{Beltrami Singularities}
Theorem \ref{thm:chiral_characterisation} gives necessary and sufficient conditions for a singularity to occur in a singular contact structure. As we remarked in \S\ref{sec:singular_contact_structures} these conditions are necessary, but not sufficient, for a singularity to occur in a singular Beltrami field. In this section we will refine Theorem \ref{thm:chiral_characterisation} to give a characterisation of Beltrami fields. Following this, we will given an example of a singuarity that is chiral but not Beltrami, showing the distinction is meaningful. First however, we present some examples of singularities that are Beltrami.
\begin{example} \label{ex:morse_zero}
  The Morse index 1 singularity, defined by the homogeneous polynomial $\phi = x^2/2 + y^2/2 - z^2$, is harmonic with respect to the Euclidean metric. Following the approach in Theorem \ref{thm:chiral_characterisation}, we obtain a singular contact form with this type of singularity,
  \begin{equation}
    \eta_t = (x+tyz)dx + (y-txz)dy - 2zdz.
  \end{equation}
  For any $t \neq 0$ this is a germ of singular contact form with a Morse index 1 zero. $-\eta_t$ gives the form for a Morse index 2 zero. We see that the order of the contact jet is 2, one higher than the order of the sufficient jet.

  Let $t > 0$, so $\eta_t$ is a positive singular contact form. We can find a Riemannian metric $g_t$ such that $\star_{g_t} \eta_t = d\eta_t$, for example the one defined by the matrix
  \begin{equation}
    g_t = t^{2/5}\begin{bmatrix}1 & 0 &\frac{t}{2}y \\ 0 & 1 & -\frac{t}{2}x \\ \frac{t}{2}y & -\frac{t}{2}x & 1 + \frac{t^2}{4}(x^2+y^2) \end{bmatrix},
  \end{equation}
  with respect to the coordinate basis on $\mbb{R}^3$. Thus the chiral Morse singularities are also Beltrami.
\end{example}
\noindent From this example, we conclude that in the generic case there is no distinction between chiral and Beltrami singularities, and therefore distinctions can only arise when one allows degenerate singularities.
\begin{proposition} \label{prop:morse_beltrami}
  A Morse singularity is Beltrami if and only if it is chiral.
\end{proposition}
Even for applications to fluid dynamics it is necessary to consider certain degenerate singularities, as their unfoldings, decompositions of the singularity into multiple simpler singularities, describe how singularities are formed from a previously nonsingular field via a homotopy. This is seen in a famous class of Beltrami fields.
\begin{example}
  An important class of Beltrami fields are Arnold--Beltrami--Childress (ABC) fields on $\mbb{T}^3$, which are eigenvalues of the curl operator (in the Euclidean metric on $\mbb{T}^3$) with eigenvalue 1. They are defined by
  \begin{multline}
    V_{ABC} = (A\sin z + C\cos y)\partial_x + (B\sin x +A\cos z) \partial_y + (C\sin y+B\cos x)\partial_z,
  \end{multline}
  for constants $A, B, C \ge 0$. Without loss of generality we can assume $1=A \ge B \ge C \ge 0$, and then the vector field is nonsingular if and only if $B^2+C^2 < 1$~\cite{dombre86}. When $B^2 + C^2 = 1$, the fields contain at least one degenerate singularity of index 0. In the notation of Arnold~\cite{arnold88} these singularities are of type $A_2$. When $B^2 + C^2 > 1$ there are singularities of Morse type which have either Morse index 1 or Morse index 2. By performing a Taylor series expansion around the zeros, we can check that these have exactly the local structure given in Example \ref{ex:morse_zero}; the calculation is straightforward and is carried out in~\cite{dombre86} for the case $A=B=C=1$.

  This example shows that the singularity germ $A_2$ defined by a function of the form $\phi = x^3 + y^2 - z^2$ is also Beltrami. More explicitly, for any $a \in \mbb{R}$,
  \begin{equation}
    \phi = \frac{1}{3}x^3 - (1-a)xy^2 - axz^2 + \frac{1}{2}y^2 - \frac{1}{2}z^2
  \end{equation}
  is a function defining the $A_2$ that is harmonic with respect to th Euclidean metric.

\end{example}
\noindent In order to study Beltrami singularities, let us show how to produce a Riemannian metric compatible with a singular contact structure from a Riemannian metric for which the underlying singularity is harmonic. This is a refinement of Theorem \ref{thm:chiral_characterisation}.
\begin{theorem} \label{thm:compatible_metric_from_harmonic}
  A singularity $\phi$ is intrinsically harmonic with respect to some Riemannian metric if and only if it is Beltrami. Moreover, if $\phi$ is harmonic with respect to $g$, then there exists a germ of a analytic singular contact form $\eta = d\phi + t\nu$, with $\nu \in I_{d\phi}$ and $\star_{t^{-2}g} \eta = d\eta$, for $t$ sufficiently small.
\end{theorem}
\begin{proof}
  Recall that rescaling a Riemannian metric $g$ by a positive constant $\lambda$ rescales the norm by a factor of $\lambda$, the star operator by a factor of $\lambda^{-1/2}$, and the volume form by a factor of $\lambda^{3/2}$.

  First, suppose $\phi$ is Beltrami. Since $\phi$ is chiral, there is a germ of a singular star operator $\star$ such that $\star d\phi = d\nu$, for some $\nu$. Furthermore, as $\phi$ is Beltrami there is a germ of a Riemannian metric $g$ with volume form $\mu$ such that $\eta = d\phi + t\nu$ is singular contact and $\star_g \eta = d\eta = td\nu$; note that we need not have $\star_g = \star$.

  By restricting to a closed ball around the origin, we may write $\eta = d\psi + \delta \alpha + \gamma$ via the Hodge decomposition theorem applied to $g$, where $\psi$ is a function, $\alpha$ a 2-form, and $\gamma$ a harmonic 1-form. Then $\nu = d(\psi-\phi) + \delta_\alpha + \gamma$, from which we conclude that $t d\star_g \nu = d\star_g d\psi- d\star_g d\phi$. We also compute that $\star_g d\phi = t(d\nu - \star_g \nu)$, from which we conclude that $d\star_g d\psi = 0$, so $\psi$ is harmonic with respect to the Riemannian metric $g$. However, $\psi$ and $\phi$ are evidently equivalent as germs of functions, and since being intrinsically harmonic is a property of the equivlanence class and not its representatives, we conclude that $\phi$ is intrinsically harmonic.

  Conversely, suppose $g$ is a germ of a Riemannian metric for which $\phi$ is harmonic. Restricting to a closed neighbourhood around the origin, we may again invoke the Hodge decomposition theorem and choose a coclosed $\nu_1$ such that $\star_g d\phi = d\nu_1$. Since $\nu_1$ is coclosed, there exists $\nu_2$ such that $\star_g \nu_1 = d\nu_2$. Iterating this process, we obtain a family of 1-forms $\nu_j$ with $\star_g \nu_j = d\nu_{j+1}$. Now define a 1-form via a formal power series in a parameter $t$,
  \begin{equation}
    \eta = d\phi + \sum_{j=1}^\infty t^j \nu_j.
  \end{equation}
  On our closed ball we can assume there is a constant $M$ bounding the components $\nu_j$ for each $j$; that we can choose a single constant independent of $j$ follows from the relations between the $\nu_j$. Then for $t$ sufficiently small we can conclude the sum converges absolutely and hence converges, so $\eta$ is well-defined. Then
  \begin{equation}
    d\eta = \sum_{j=1}^\infty t^j d\nu_j = t \star d\phi + t^2 \star d\nu_1 + t^3 \star d\nu_2 + \dots = t\star_g \eta.
  \end{equation}
  Thus, $\star_{t^{-2}g} \eta = d\eta$, implying that $\eta$ is singular contact and that $\phi$ is Beltrami.
 \end{proof}
\noindent The proof of this theorem, applied to the particular case where the singularity is harmonic with respect to the Euclidean metric, shows us how to construct a vector field containing this singularity that is a Beltrami field on Euclidean space.
\begin{corollary}
  If a singularity has a representative $\phi$ that is harmonic with respect to the Euclidean metric, then $\phi$ can occur as a singularity of a Beltrami field on Euclidean $\mbb{R}^3$. It follows that a singularity of any index can occur in a Beltrami field in Euclidean space.
\end{corollary}
\begin{proof}
  The first claim is immediate from Theorem \ref{thm:compatible_metric_from_harmonic}: $\nabla \phi$ can be perturbed into a vector field that is Beltrami with respect to the Euclidean metric rescaled by a factor of $t^{-2}$, but a change of coordinates allows us to assume $t=1$. For the second claim, we observe that we can produce a singularity of any index using spherical harmonics.
\end{proof}
\noindent It can be checked that certain chiral functions are harmonic with respect to some star operator that degenerates at the singular point, but not with respect to any genuine Riemannian metric. Thus, not all chiral singularities are Beltrami. We conclude by giving examples of such degenerate singularities. Define the \textit{corank} of a singularity to be the corank of the Hessian matrix of any representative function $\phi$ evaluated at the origin, that is, the number of zero eigenvalues; this does not depend on the choice of representative.
\begin{proposition} \label{prop:corank2_not_beltrami}
  Let $\phi$ be a germ of a singularity with corank 2. Then $\phi$ is not Beltrami.
\end{proposition}
\begin{proof}
  The condition on the corank of the singularity implies that, no matter what representative of the germ we choose, the trace of the Hessian matrix will be nonzero at the origin.

  Suppose for a contradiction that $\phi$ is Beltrami. Then, by Theorem \ref{thm:compatible_metric_from_harmonic}, there exists a germ of a Riemannian metric $g$ for which $\phi$ is harmonic. We can decompose the star operator of $g$ as $\star_g = \star_0 +\star_1$, where $\star_0$ is a symmetric and positive-definite bilinear form (the `constant part' of the star operator) and $\star_1$ is singular. Since $\star_0$ is constant it is the star operator of a germ of a flat Riemannian metric on $\mbb{R}^3$, and we can find local coordinates in which the matrix for $\star_0$ has entries $(\star_0)_{ij} = \delta_{ij}$.

  Now consider $d\star_g d\phi = d\star_0d\phi + d\star_1 d\phi$. As the corank of $\phi$ is 2, the first term on the right-hand side must contain a nonzero constant term, since it is equal to the trace of the Hessian of $\phi$. However, the second term on the right-hand side contains terms of at least linear order. Restricting to the origin, we have $d\star_g d\phi|_{p=0} = d\star_0 d\phi|_{p=0} \neq 0$. This implies that $d\star_g d\phi \neq 0$, a contradiction.
 \end{proof}
\noindent A corollary of the argument in the proof is that
\begin{corollary}
  A necessary condition for a singularity to be Beltrami is that there be a representative function $\phi$ such that the trace of the Hessian of $\phi$ vanishes at the origin.
\end{corollary}
\noindent The corank 2 singularities do not meet this condition, and nor do functions with spherical level sets. We do not know if this condition is sufficient. It does allow us to describe which among the simple singularities can occur in a Beltrami field.
\begin{corollary}
  Among the simple singularities, germs from the families $D^\pm_k, E_6, E_7, E_8$ (in the notation of~\cite{arnold88}) are never Beltrami. Function germs of type $A_k$ defined by $\phi = \pm(x^{k+1} + y^2 + z^2)$ are not Beltrami, and are not even chiral when $k$ is odd.
\end{corollary}
\begin{proof}
  All of the germs given in the statement have a Hessian matrix with non-vanishing trace, and hence Proposition \ref{prop:corank2_not_beltrami} implies they are not Beltrami. For singularities of type $A_k$, $k$ odd, defined by $\phi = \pm(x^{k+1} + y^2 + z^2)$, the level sets of the function are spheres and hence, by Corollary \ref{cor:no_spheres}, they are not chiral.
\end{proof}
\begin{example}
  We give an example of a germ that is chiral but not Beltrami. The singularity $D_4^-$ can be defined by the function germ $\phi = x^2y - y^3/3 +z^2/2$ and has index $-2$. The previous corollary implies it is not Beltrami, however we can construct $\nu$ such that the germ $\eta = d\phi + t\nu$ is a germ of a singular contact form. The perturbation is~\cite{pollard19}
  \begin{equation}
    \nu = (z(x^2-y^2)-yz^3)dx + (xz^3 -2xyz)dy.
  \end{equation}
  We conclude that while the singularity $D^-_4$ is chiral, and can be realised in a cholesteric droplet~\cite{posnjak17,pollard19}, it may not be realised in any Beltrami field. A singularity of index $-2$ in a Beltrami field can be produced using spherical harmonics; the simplest example using spherical harmonics yields a singularity equivalent to the germ $U_{12}$ (in the notation of~\cite{arnold88}) which can be defined by a function germ $\phi = x^2y - y^3 + z^4 + axyz^2$, for a parameter $a \in \mbb{R}$. This is of interest, because it implies that the index $-2$ singularities of lowest multiplicity, \textit{a priori} the most likely to occur, in fact cannot occur in Beltrami fields, and one is forced to have singularities of higher multiplicicty and therefore more complicated structure. It also impacts on the possible unfoldings of the singularity, for instance $U_{12}$ can be decomposed into three singularities of type $D^-_4$, but this decomposition is not available in a Beltrami field.
\end{example}

\section{Global Properties of Singular Contact Structures} \label{sec:global_properties}
Although Theorem \ref{thm:chiral_characterisation} is in general only a local result, when $M$ has vanishing first homology the same argument used in the proof gives a global result on the perturbation of closed 1-forms with singularities into singular contact structures.
\begin{proposition} \label{prop:global_perturbation}
  Let $M$ be a 3-manifold (not assumed to be closed) with $H_1(M) = 0$ and let $\phi$ be a globally defined function such that $d\phi$ has finitely many isolated zeros at the points of a set $\Sigma$. Suppose further that $d\phi$ is harmonic with respect to either a Riemannian metric on $M$, or at least a Riemannian metric on $M-\Sigma$ that is undefined on $\Sigma$ but whose star operator extends. Then $d\phi$ can be linearly perturbed into a singular contact form $\eta = d\phi + t\nu$.
\end{proposition}
\noindent It is interesting to observe that Proposition \ref{prop:global_perturbation} yields, via Calabi's theorem, a proof that every nonsingular foliation by a closed 1-form on a closed 3-manifold $M$ with $H^1(M) = 0$ has a compact leaf; this is Novikov's Theorem~\cite{novikov64}. By Calabi's theorem, a closed 1-form $\alpha$ is intrinsically harmonic if and only if every leaf intersects a closed transversal curve. On a closed manifold $M$, a closed nonsingular 1-form $\alpha$ cannot be linearly perturbed into a contact form, and therefore cannot be intrinsically harmonic. Suppose for a contradication that it could. Then there exists some 1-form $\beta$ such that $\alpha \wedge d\beta > 0$. However, Stokes' theorem implies that
\begin{equation}
  \int_M \alpha \wedge d\beta = -\int_M \beta \wedge d\alpha = 0.
\end{equation}
We conclude that are there no nonsingular, closed, intrinsically harmonic 1-forms on $M$. By another classical result in foliation theory, any noncompact leaf intersects a closed transveral (see Lemma 3.16 of~\cite{moerdijk03}). Thus we conclude that the foliation defined by $\alpha$ must have a compact leaf.

In the next sections we will discuss global perturbations of singular plane fields into a singular contact structures in greater generality.

\subsection{Global Existence of Singular Contact Structures} \label{sec:homotopy_invariants}
In general, it is not possible to linearly perturb a singular plane field into a singular contact structure as in Proposition \ref{prop:global_perturbation}. In this section we will prove several theorems about global perturbations of singular plane fields into singular contact structures on closed manifolds, and on manifolds with boundary where the boundary condition is prescribed. There is some overlap between our results in this section and those of~\cite{miranda18}, although we adopt a different perspective.

To prove our extension results, we need to consider the invariants of homotopy classes of plane fields, see {\sl e.g.} \cite{geiges08}. Let $M$ be a closed 3-manifold and $\xi$ a plane field on $M$. Fix a trivialisation of the tangent bundle and a Riemannian metric. Then $\xi$ can be identified with a vector field $N$, its unit normal. In turn, $N$ determines a map $f_\xi : M \to S^2$. Choose a regular value $p$ of this map. The preimage of $p$ is a closed link $L_\xi$ in $M$, which inherits a framing by pulling back a basis for the tangent space $T_p S^2$.

A homotopy of plane fields $\xi_t$ induces a homotopy through the maps $f_{\xi_t}$, and in turn changes the associated links $L_{\xi_t}$ via a framed cobordism. Using the Pontryagin--Thom construction, we can associate a plane field to every framed link, and then a cobordism between these links induces a homotopy between the associated plane fields. This is a sketch of the proof of a well-known result.
\begin{theorem}
  Homotopy classes of plane fields on a closed 3-manifold $M$ are in a one-to-one correspondence with cobordism classes of framed links.
\end{theorem}

The correspondence between homotopy classes of plane fields on a manifold $M$ and cobordism classes of links also carries over to the situation where $M$ is a compact 3-manifold with boundary. Fix a trivialisation of $TM$. Then we can identify unit vector fields $N$ with maps $f_N : M \to S^2$. Suppose $\partial M$ is connected. If we have such a map defined near the boundary of $M$, then it extends over the interior if and only if it is homotopic to a constant map. When the boundary of $M$ is not connected, $f_N$ need not be homotopic to a constant in order to extend over the interior without singularities; for example the radial vector field on $S^2 \times [0,1]$ is not homotopic to a constant, but obviously extends from the boundary into the interior.

In the more general situation where the boundary of $M$ has multiple components $C_i$, we can define the \textit{degree} of the map $f_N$ to be the sum of the degrees of the maps $f_N |_{C_i}$, taken over all components. By a slight abuse of terminology, we will say a plane field has degree zero if the map $f_N$ it induces has degree zero in some trivialisation.

\begin{lemma} (\cite[Lemma 4.1]{etnyre13}) \label{lemma:pontryagin-thom}
  Let $\xi_0$ be a plane field defined near $\partial M$ with degree 0. There is a one-to-one between homotopy classes of plane fields $\xi$ on $M$ that extend $\xi_0$ and the set of framed links in the interior of $M$ up to framed cobordism.
\end{lemma}
\noindent A fundamental result, the Lutz--Martinet theorem~\cite{geiges08}, establishes that on a closed 3-manifold $M$ there is a contact structure in every homotopy class of plane fields. The proof of this result involves constructing a single contact structure, and then using a type of surgery called the Lutz twist. We discuss the Lutz twist as an operation on singular contact structures in \S\ref{sec:lutz_twist}. First, we prove a singular analogue of the Lutz--Martinet theorem.

Lemma \ref{lemma:pontryagin-thom} leads to the following extension result for contact structures defined in the neighbourhood of the boundary of a compact 3-manifold $M^\prime$ that embeds in a 3-manifold $M$ without boundary.
\begin{lemma} \label{lemma:contact_extension}
  Suppose $\xi_0$ is a contact structure defined in a neighbourhood of $\partial M^\prime$ with degree 0. There is a contact structure $\xi$ defined on all of $M^\prime$ that agrees with $\xi_0$ near the boundary.
\end{lemma}
\begin{proof}
  By Lemma \ref{lemma:pontryagin-thom}, there exists a plane field $\xi^\prime$ that extends $\xi_0$. We need only show that $\xi^\prime$ is homotopic (rel. boundary) to a contact structure. First observe that there is a contact structure on the open manifold $W = \text{int}(M^\prime)$: by assumption there is an embedding $j : W \to M$, and we can pull back some contact structure on $M$ along this embedding to give a contact structure on $W$. Moreover, we can produce a contact structure in every homotopy class of plane fields on $W$ using the same techniques as in the Lutz--Martinet construction, the details of which are given in~\cite{geiges08}. Therefore, $\xi^\prime$ is homotopic rel. boundary to a contact structure $\xi$, as required.
 \end{proof}
Next we have our version of the Lutz--Martinet theorem.
\begin{theorem} \label{thm:homotopy_singular_contact}
  Let $M$ be a closed 3-manifold. Suppose a singular plane field $\xi$ has only chiral singularities, that is, in a neighbourhood $U$ of each singular point $p$ there is a diffeomorphism $f$ such that $\xi$ is defined by the kernel of the 1-form $f^* d\phi$, where $\phi$ is a chiral function. Then $\xi$ is homotopic to a singular contact structure via a homotopy that fixes the singular points.
\end{theorem}
\begin{proof}
  We can perform the desired homotopy in a neighbourhood of each singular point to produce a contact region about each singular point. Remove an open ball $B_i$ around each singular point $p_i$ that is contained entirely in the contact region. This leads to a manifold $M^\prime = M - \bigcup_i B_i$ with boundary $\partial M^\prime$, and plane field on $M^\prime$ that is contact near the boundary. This plane field has degree $0$, since the unit normal to the original plane field on $M$ satisfied the Poincar\'e--Hopf theorem. By Lemma \ref{lemma:pontryagin-thom} we can homotope the plane field into a contact structure on $M^\prime$, leaving the boundary fixed. Glueing the balls $B_i$ back in completes the proof.
 \end{proof}
\noindent The proof of this theorem is unfortunately nonconstructive. For applications it would be useful to have a more constructive approach. The geometric heat flow method of proving the existence of contact structures on closed 3-manifolds developed by Altschuler~\cite{altschuler95} may be useful in this regard. In Altschuler's approach, one begins with a confoliation and regards the contact part as being `hot'. A heat diffusion equation is then used to diffuse the heat along the leaves of the foliated part. As long as every point in space can be connected to the hot region by a finite length path tangent to the leaves of the foliated part, then the entire manifold will instantaneously become hot: the confoliation will have become a contact structure.

Suppose we have a foliation defined by a closed intricially harmonic 1-form with isolated singularities, which will necessarily be chiral by Theorem \ref{thm:chiral_characterisation}. Then, by Calabi's theorem, there exists a finite link $L$ that intersects every leaf of the foliation. We can perturb the foliation into a confoliation in a neighbourhood of $L$ using the method developed in \S5 of~\cite{altschuler95}. As $L$ passes through every leaf of the foliation, this `hot region' is connected to any `cold' point by a path tangent to some leaf of the foliated part. Then turning on the heat equation will diffuse the heat and instantly make the plane field contact at every point except for the singularities.

We can also use Lemma \ref{lemma:singular_contact_extension} to study singular contact structures on bounded submanifolds of $\mbb{R}^3$ where the boundary behaviour is prescribed; this is the natural setting for experiments on liquid crystal systems. Let $M \subset \mbb{R}^3$ be a compact domain.
\begin{lemma}  \label{lemma:singular_contact_extension}
  Let $\xi_0$ be a contact structure near $\partial M$ whose Reeb-like fields $R$ are transversal to the boundary and point outwards. Then there exists a singular contact structure $\xi$ on $M$ that extends $\xi_0$.
\end{lemma}
\begin{proof}
  As $R$ is nonsingular and point out from the boundary there is a singularity `deficit' of $\frac{1}{2}\chi(\partial M)$ preventing us from extending $\xi_0$. Fix disjoint closed balls $B_i$, $i = 1, \dots, \frac{1}{2}|\chi(\partial M)|$ inside $M$, and let $W = \cup_i B_i$. When $\chi(\partial M)$ is negative, we place a single singularity of index $-1$ in each ball, while if $\chi(\partial M)$ is positive, we use a singularity of index $+1$. The 1-form $\eta_t$ given in Example \ref{ex:morse_zero} defines a singular contact structure on each $B_i$ that has Morse singularity of index $-1$ at the centre, while $-\eta_t$ will give a singularity of index $+1$. This gives us a singular contact structure $\zeta_0$ on $W$.

  Now remove the interior of $W$ from $M$ to give a manifold $M^\prime$ with boundary $\partial M^\prime = \partial M \cup \partial W$, which has a contact structure defined near its boundary, equal to $\zeta_0$ on the $\partial W$ part and $\xi_0$ on the $\partial M$ part. Let $N$ be the unit normal to this contact structure. $N$ is transverse to the boundary of $M^\prime$ and the degree of the map $f_N$ is zero, so by Lemma \ref{lemma:contact_extension} we can extend this to a contact structure on all of $M^\prime$ which agrees with the contact structure we had near the boundary. Glueing the interior of $W$ back in yields a singular contact structure on $M$ which agrees with $\xi_0$ near $\partial M$.
 \end{proof}
\noindent Although the proof of Lemma \ref{lemma:singular_contact_extension} creates a plane field with $\frac{1}{2}|\chi(\partial M)|$ Morse singularities, all that is necessary is that the total index of the singularities be $\frac{1}{2}\chi(\partial M)$ and that each one is chiral; we could include additional Morse singularities, and also degenerate singularities if we wish. If $R$ points into rather than out of the boundary, we reverse signs through the rest of the proof.

\subsubsection{Normal Anchoring}
A typical boundary condition for liquid crystal experiments is to impose normal anchoring, such that the director is fixed to be the unit normal to the boundary. It is natural to ask whether a director satisfying this boundary condition can be the normal to a singular contact structure everywhere in the interior. Let $H_g$ denote the handlebody of genus $g \ge 0$. Using Proposition 1.3.13 of Eliashberg \& Thurston~\cite{eliashberg97}, we deduce the following theorem.
\begin{theorem} \label{thm:contact_tangent}
  There is a singular plane field on $H_g$, $g > 0$ that is a singular contact structure on the interior of $H_g$ and is tangent to the boundary. No such plane field exists for $g=0$.
\end{theorem}
\begin{proof}
  By Proposition 1.3.13 of~\cite{eliashberg97}, for $g > 0$ there exists a plane field in a neighbourhood of $\partial H_g$ that is tangent to $\partial H_g$ and contact away from the boundary. By removing a smaller neighbourhood of the boundary, we have produced a contact structure near to the boundary of $H_g$. Using this plane field as $\xi_0$ in Lemma \ref{lemma:singular_contact_extension}, we can extend it to a singular plane field on all of $H_g$ with the desired properties. As we remarked in \S\ref{sec:local_properties}, when $g=0$ no such singular plane field can exist by the Reeb stability theorem for confoliations.
 \end{proof}
\noindent Note that Theorem \ref{thm:contact_tangent} remains true for a 3-manifold $M$ with multiple boundary components, provided that they are all surfaces of positive genus.
\begin{example}
  An example of this boundary behaviour in a neighbourhood $T^2 \times [0,1]$ of the boundary of $H_1$ is given by the 1-form
  \begin{equation}
    \eta = \cos(z)dz + \sin(z)(\cos(z)dx + \sin(z)dy)
  \end{equation}
  where $x,y$ are coordinates on $T^2$ and $z$ is the coordinate on $[0,1]$. The plane field this 1-form defines is tangent to $T^2 \times 0$ and contact everywhere else.
\end{example}
\noindent Theorem \ref{thm:contact_tangent} applies to experiments on handlebody droplets with normal anchoring~\cite{ellis18,musevic17,pairam13}, and shows that for such systems the liquid crystal can be twisted with a single handedness everywhere in the interior, provided the genus is positive. On the other hand, for spherical droplets (or more generally domains with a spherical boundary component) there must always be a region with the wrong/opposite handedness near the boundary~\cite{pollard19,ackerman2016}. Note that the theorem does not prescribe the geometry of the region of the region of reversed twist, only its existence. Typically, cholesteric droplets exhibit boundary defects of Morse index 0, which lie on the interface between regions of opposite handedness. The region of reverded twist is then spherical; toroidal regions of reversed handedness without the presence of boundary defects have also be observed~\cite{posnjak17,pollard19}.

\subsubsection{Planar Anchoring}
The other boundary condition commonly imposed in experiments is planar anchoring, when the director is tangent to the boundary, which is also the natural boundary condition for a Beltrami field describing a fluid flow in a bounded domain. For a boundary of genus $g \neq 1$ this will necessarily result in singularities on the boundary.

To understand this situation, we must study singular contact structures whose Reeb-like fields are tangent to some surface $S$. We will use techniques from contact topology that analyse contact structures in a neighbourhood of a surface. Let $S$ be an oriented surface embedded in a contact manifold $M$ with contact structure $\xi$. The \textit{characteristic foliation} on $S$ is the singular foliation of $S$ defined by $\xi_S = TS \cap \xi$. The characteristic foliation on $S$ determines the contact structure in a neighbourhood of $S$ up to homotopy~\cite{geiges08}.

Choose an area form $\Omega$ on $S$ and consider a neighbourhood $S \times [-1,1]$. Let $\mc{F}$ be a foliation on $S$ with isolated singularities of winding $\pm 1$, directed by a vector field $X$. Let $\Sigma = \{p_1, \dots p_n \}$ denote the set of singular points of $X$. We can decompose $\Sigma = \Sigma^0 \cup \Sigma^1$, where $\Sigma^0$ consists of those points in $\Sigma$ for which the divergence of $X$ with respect to $\Omega$ vanishes, and $\Sigma^1$ contains the remainder of the singular points. Note that the vanishing of the divergence is independent of the choice of area form $\Omega$ and the choice of the vector field $X$ directing the foliation, and therefore this decomposition is also independent of the various choices. A foliation $\mc{F}$ can be realised as the characteristic foliation induced on $S$ by a contact structure if and only if the set $\Sigma^0$ is empty~\cite{geiges08}.

The following lemma gives a construction of singular contact structures in a neighbourhood of a surface. The proof is an adaptation of the standard construction of a contact form inducing a particular characteristic foliation on a surface.
\begin{lemma} \label{lemma:singular_contact_construction}
  Let $X$ be a vector field with singularity set $\Sigma$. Suppose there exists a vector field $Y$ transverse to $X$ except at the singular points such that the divergence of $Y$ is nonzero on $\Sigma^0$. There exists a singular contact form $\eta$ on $S \times [-1,1]$ such that the characteristic foliation induced by $\eta$ on $S \times 0$ is directed by $X$ and the singular points of $\eta$ occur only at the points of $\Sigma^0 \subset S \times 0$.
\end{lemma}
\begin{proof}
  Let $z$ be the coordinate on $[-1,1]$. Define $\beta = \iota_X \Omega$. Then $d\beta = \mc{L}_X \Omega = \text{div}(X)\Omega$. By assumption there is a 1-form $\gamma = \iota_Y \Omega$ such that $\beta \wedge \gamma \ge 0$, with equality only at the points where $X = 0$. Let $X_z = X - zY$, $u_z = \text{div}(X_z)$ be 1-parameter families of vector fields and functions respectively.

  Define $\beta_z = \beta + z(du_z - \gamma)$ and $\eta = \beta_z + u_zdz$. We compute that $(\eta \wedge d\eta)|_{z=0} = (u^2_z \Omega + \beta \wedge \gamma) \wedge dz \ge 0$, with equality at the points where both $\beta$ and $u$ vanish. It follows that, in a neighbourhood of $S \times 0$, $\eta \wedge d\eta = 0$ only at the points $\Sigma^0$. At such points $\eta$ itself vanishes, so these are the singularties of $\eta$. It follows that $\eta$ is a positive singular contact form.
 \end{proof}
\noindent The assmptions of the lemma hold when, for example, the Jacobian matrix of $X$ is nondegenerate at the singular points, the generic situation. By taking $u_z = \text{div}(X + zY)$, $\beta_z = \beta - z(du_z - \gamma)$, and $\eta = \beta_z - u_zdz$ instead, we produce a negative singular contact form.

\begin{example}
  The lemma offers an alternative strategy for constructing the contact jets of the chiral Morse singularities. Consider the manifold $M = \mbb{R}^2 \times [-1,1]$ with area form $\Omega = dx \wedge dy$ on $\mbb{R}^2$ and the vector field $X = y\partial_x -x\partial_y$ on the set $z=0$. Define $\beta = \iota_X \Omega = xdx + ydy$ and $\gamma = xdy - ydx$, which satisfies $\beta \wedge \gamma = (x^2+y^2)dx \wedge dy$. We have $Y = x\partial_x + y\partial_y$. The divergence of $X$ vanishes but the divergence of $Y$ is equal to $2$. Set $X_z = X - zY$ and $u_z = -2z$. Following the rest of the construction, we have $\beta_z = \beta - z\gamma = (x+yz)dx + (y-xz)dy$ and $\eta = \beta_z -2zdz$. We have rediscovered the normal form of the Morse index 1 singularity given in Example \ref{ex:morse_zero}.
\end{example}
\noindent As a counterpart to the result that the divergence does not vanishing at any singular point of the characteristic foliation induced by a contact structure, we have the following lemma about the characteristic foliations induced by singular contact structures.
\begin{lemma} \label{lemma:characteristic_foliation_characterisation}
  Let $X$ be a vector field on an oriented surface $S$ with singular set $\Sigma$. The following are equivalent:
  \begin{enumerate}
    \item There exists a singular contact form $\eta$ on $S \times [-1,1]$ with Reeb-like field $R$ that is tangent to $S \times 0$, has singularities only on $S \times 0$ that coincide with with singularities of $X$, and induces a characteristic foliation on $S \times 0$ directed by $X$.
    \item $X$ is divergence-free for some area form $\Omega$.
  \end{enumerate}
\end{lemma}
\begin{proof}
  Fix an area form $\Omega$. First we will show that $\text{div}(X) = 0$ is a necessary assumption. Suppose we have a singular contact form $\eta$ with Reeb-like field $R$ that is tangent to $S$. In a neighbourhood of $S$ that is diffeomorphic to $S \times [-1,1]$ we can write $\eta = \beta_z + u_z dz$, where $\beta_z$ is a family of 1-forms on $S$ and $u_z : S \to \mbb{R}$ a family of functions. We can similarly write $R = R_S + R_z \partial_z$, where $R_S(z)$ is a vector field tangent to $S \times z$, and $R_z = 0$ at $z=0$. We compute
  \begin{equation}
    d\eta = \text{div}(X_z)\Omega + dz \wedge \dot{\beta}_z + du_z \wedge dz,
  \end{equation}
  where $X_z$ directs the characteristic foliation on $S \times z$, and $\dot{\beta}_z$ denotes the $z$ derivative of the family $\beta_z$. The Reeb-like field satisfies $\iota_R d\eta = 0$, that is
  \begin{equation}
    0 = \text{div}(X_z)\iota_{R_S}\Omega + R_z \dot{\beta}_z + (du(R) - \dot{\beta}_z(R))dz - R_z du_z.
  \end{equation}
  Choose local coordinates $x,y$ on $S$, and write $R_S = R_x \partial_x + R_y \partial_y$. Then the Reeb-like field satisfies the following three equations on $S \times [-1,1]$
  \begin{equation} \label{eq:reeb_conditions}
    \begin{aligned}
      \text{div}(X_z)R_y = R_z(\partial_x u_z + X_y), \\
      \text{div}(X_z)R_x = R_z(X_x -\partial_y u_z), \\
      du_z(R) = \dot{\beta}_z(R).
    \end{aligned}
  \end{equation}
  Restricting to $z=0$ where $R_z$ vanishes, we see at each point on $S \times 0$ we must have either $R_y = R_x = 0$ or $\text{div}(X) = 0$. If $R_x, R_y$ are both zero then we are at a singular point, and these are isolated. It follows that we need $\text{div}(X) = 0$ everywhere except at a finite set of points, however continuity then implies it must vanish at those points as well.

  Conversely, given a divergence-free $X$ we can construct a singular contact form $\eta$ on $S \times [-1,1]$ inducing that characteristic foliation on $S \times 0$ using Lemma \ref{lemma:singular_contact_construction}. The Reeb-like fields satisfy (\ref{eq:reeb_conditions}). Since the divergence of $X$ vanishes, at each point of $S \times 0$ we either have $R_z = 0$ or $\partial_x u_z + X_y = X_x - \partial_y u_z = 0$. The construction given in Lemma \ref{lemma:singular_contact_construction} is such that $u_z$ and it derivatives vanish along $S \times 0$. Therefore, we have $R_z = 0$ away from the singular points of $X$, and since these are isolated we again conclude that $R_z=0$ everywhere on $S \times 0$, i.e. $R$ is tangent to $S \times 0$.
 \end{proof}

\begin{theorem}
  For any $g \ge 0$, there exists a singular contact structure on $H_g$ whose Reeb-like fields are tangent to the boundary.
\end{theorem}
\begin{proof}
  Choose a symplectic form $\Omega$ and a Morse function $\phi$ on $\partial H_g$. Define $X$ by $\iota_X \Omega = d\phi$. Then $\text{div}(X) = d\iota_X \Omega = 0$. By Lemma \ref{lemma:characteristic_foliation_characterisation} there is a singular contact structure in a neighbourhood of the boundary that only has singularities only on $\partial H_g$. We extend over the interior as in Lemma \ref{lemma:singular_contact_extension}. No additional singularities are required in the interior.
 \end{proof}
\noindent This theorem, and the construction methods given in the preceeding two lemmas, can be immediately applied to the study of defects in cholesteric liquid crystals shells, where the material domain is $S^2 \times [0,1]$ and we impose tangential anchoring on the boundary~\cite{darmon16}. The problem of determining the possible directors in such a shell with defects only on the boundary reduces to the study of families of foliations $\mc{F}_t$, $t \in [0,1]$, on the 2-sphere, where for $t=0,1$ the foliation must be directed by a divergence-free vector field, and for $t \in (0,1)$ the divergence of a vector field directing the characteristic foliation must not vanish at any singular points.

Lemma \ref{lemma:characteristic_foliation_characterisation} also leads to the following proposition that shows we may modify a contact form in the neighbourhood of a surface so as to introduce singularities on that surface, and moreover that we can do it in such a way that the Reeb-like fields are tangent to the surface. First, we recall another notion from contact topology. A vector field is called a \textit{contact vector field} if it preserves the contact structure. A surface is called \textit{convex} if there is a contact vector field transverse to it. Convex surfaces are generic: given an embedded surface $S$, there is a convex surface $S^\prime$ close to $S$~\cite{geiges08}. Moreover, the characteristic foliation induced on a convex surface is always directed by a Morse--Smale vector field, and therefore has nondegenerate singular points. The following result is also proved in~\cite{miranda18}.
\begin{proposition}
  Let $\eta$ be a contact structure on $M$ and $S \subset M$ a closed surface. There is a surface $S^\prime$ close to $S$ and a singular contact form on $M$ that agrees with $\eta$ outside of an open neighbourhood of $S^\prime$ and has Reeb-like fields tangent to $S^\prime$.
\end{proposition}
\begin{proof}
  Let $Y$ be a vector field directing the characteristic foliation induced on $S$ by $\eta$. By perturbing $S$ to a nearby convex surface $S^\prime$ if necessary, we may assume that the Jacobian of $Y$ is nonzero at the singular points of $Y$, a consequence of the fact that the characteristic foliation is Morse--Smale. There is a vector field $X$ that is transverse to $Y$ everywhere except at its singular points and which has vanishing divergence at the singular points. On a neighbourhood $S^\prime \times [-1,1]$ of $S^\prime$, take $X_z = (1-z^2)X - zY$ and construct a singular contact form on $S^\prime \times [-1,1]$ as in Lemma \ref{lemma:singular_contact_construction}. The resulting singular contact form is homotopic to $\eta$ near $S^\prime \times -1$ and $S^\prime \times +1$, and has Reeb-like fields tangent to $S^\prime \times 0$ by Lemma \ref{lemma:characteristic_foliation_characterisation}.
 \end{proof}

\subsection{Overtwistedness} \label{sec:overtwisted}
The dichotomy between tight and overtwisted contact structures is a central theme in contact topology. We now examine it in the context of singular structures. An embedded disk $D$ whose boundary is tangent to the contact structure is called an \textit{overtwisted disk} if the characteristic foliation induced on $D$ has a limit cycle. If a contact structure $\xi$ contains an overtwisted disk, then it is called \textit{overtwisted}, otherwise it is called \textit{tight}. Overtwisted contact structures satisfy an $h$-principle: two overtwisted contact structures are homotopic through contact structures if and only if they are homotopic as plane fields~\cite{eliashberg89}; therefore, subtleties in classification arise only for tight contact structures.

Let $\xi$ be a singular contact structure on $M$. Remove open balls of radius $\epsilon$ around each singularity to produce a contact structure $\xi_\epsilon$ on a new manifold $M_\epsilon$. We will say that $\xi$ is overtwisted if $\xi_\epsilon$ is overtwisted in the usual sense for all sufficiently small $\epsilon > 0$.

In contact structures overtwistedness is a property that cannot be determined by purely local considerations: the Darboux theorem implies that each point in a contact manifold has a coordinate neighbourhood on which the contact structure is defined by $\eta = dz + xdy$, which defines the standard tight contact structure on $\mbb{R}^3$. By contrast, in this section we show that singular contact structures with singularities of nonzero index are always overtwisted, and this follows from the local structure of the singularities. To prove this, we will use the theory of convex surfaces. A surface embedded in a manifold with a singular contact structure is called convex if it does not intersect any singular point and is convex in the usual sense for the contact structure defined on the complement of the singular points.

Every oriented convex surface $S$ with characteristic foliation defined by a vector field $X$ admits a \textit{dividing curve} $\Gamma$. This is a curve satisfying the following properties:
\begin{enumerate}
  \item $S-\Gamma = S^+ \cup S^-$,
  \item $\Gamma$ is transverse to the leaves of the characteristic foliation,
  \item There is an area from $\Omega$ on $S$ such that $\pm L_X \Omega > 0$ on $S^\pm$ and $X$ points transversally out of $S^+$ along $\Gamma$.
\end{enumerate}
The set $\Gamma$ consists of points for which the divergence of $X$ (with respect to some area form) vanishes. We can identify it with a set of points where the normal to the contact plane is tangent to the surface.
\begin{proposition} \label{prop:dividing_curve_normal}
  The same $\Gamma$ is isotopic to the set of points for which the normal to the contact planes (for some choice of metric) is tangent to $S$.
\end{proposition}
\begin{proof}
  Let $X$ direct the characteristic foliation on $S$, and let $\Omega$ be the area form implicit in the definition of the dividing curve, so that $\Gamma$ is the set of points where the divergence of $X$ vanishes. In a neighourhood $S \times [-1,1]$ of $S$ the contact structure is defined by a 1-form $\eta = \iota_X \Omega + udz$, where $z$ is the coordinate on $[-1,1]$ and $u : S \to \mbb{R}$ is a function which does not depend on $z$ -- this follows from the fact that $S$ is convex, see~\cite{geiges08}. The function $u$ must satisfy $u \,\text{div}(X) - X(u) > 0$. We can take $u=+1$ on $S^+$, $u=-1$ on $S^-$, and smoothly connect these functions in an annular neighbourhood of $\Gamma$ via the constructions used in~\cite[Theorem 4.8.5]{geiges08}. The function $u$ will then vanish along some curve which is isotopic to $\Gamma$, and this is exactly the set of points for which the normal to $\eta$ in the metric $dx^2 + dy^2+dz^2$ on $S \times [-1,1]$ is tangent to $S \times 0$.
 \end{proof}
\noindent The topology of the contact structure in a neighbourhood of $S$ is determined entirely by the dividing curve. In particular, the dividing curve can be used to determine the overtwistedness of the contact structure via the following theorem of Giroux~\cite{giroux01}.
\begin{theorem} (Giroux's Criterion)
  Let $S$ be a closed convex surface of genus $g$ in a contact 3-manifold with dividing curve $\Gamma$. A neighbourhood of $S$ is tight if and only if $g > 0$ and $\Gamma$ contains no contractible curves, or $g=0$ and $\Gamma$ is connected.
\end{theorem}
\noindent Applying Giroux's Crierion and the structural result we proved in \S\ref{sec:local_properties}, we obtain the following result.
\begin{theorem} \label{thm:singular_contact_structures_are_ot}
  A singular contact structure $\xi$ with a singularity of nonzero index is always overtwisted.
\end{theorem}
\begin{proof}
  Choose a sphere centred at a singularity of index $k \neq 0$. After a small perturbation we can assume the sphere is convex, and hence that it admits a dividing curve $\Gamma$. Fix a Riemannian metric and consider the unit normal to the contact planes. If we are sufficiently close to the singularity, the unit normal is close to the vector field $\nabla \phi$, where $\phi$ is a function that describes the structure of the singularity according to Proposition \ref{prop:only_gradients}. By Lemma \ref{lemma:level_set_structure} and the fact that the level sets of $\phi$ are not spheres, the set of points where $\nabla \phi$ is tangent to $S$ has $|k| + 1$ components. By Proposition \ref{prop:dividing_curve_normal} this set is isotopic to the dividing curve $\Gamma$, and therefore $\Gamma$ has $|k|+1$ components. By Giroux's Criterion, the contact structure is therefore overtwisted.
 \end{proof}
\noindent Informally, we can identify the `arms' of the singular set of the function $\phi$ with tubes of overtwisted disks. Thus every singularity of index $k \neq 0$ sits at the confluence of $|k|+1$ tubes of overtwisted disks.

By way of an example, we apply Theorem \ref{thm:singular_contact_structures_are_ot} to the ABC fields. Contact topological arguments can be used to show that any nonsingular ABC field is transverse to a tight contact structure~\cite{etnyre00}; our result then allows us to describe the remainder of the space.
\begin{corollary}
  An ABC field is tight if and only if $B^2+C^2 \leq 1$.
\end{corollary}
\begin{proof}
  Every ABC field with parameter values $B^2+C^2 < 1$ is nonsingular, and tight by the result of Etnyre \& Ghrist~\cite{etnyre00}. For certain parameter values the ABC field will have only singularities of index 0, but these must lie on the boundary of the space of nonsingular fields, and hence make up the arc $B^2 + C^2 = 1$; these are also tight. The remainder of the space consists of ABC fields transverse to singular contact structures with at least two singularities of nonzero index, to which we may apply Theorem \ref{thm:singular_contact_structures_are_ot}.
 \end{proof}

\subsection{The Singular Lutz Twist} \label{sec:lutz_twist}
Let $\Xi$ denote the space of all contact structures (in the usual sense) on a 3-manifold $M$ (not necessarily closed), and let $\Xi_\text{sing}$ denote the space of singular contact structures. Let us enlarge $\Xi$ to also include the singular contact structures, producing the space $\mf{S} = \Xi \cup \Xi_\text{sing}$. By \textit{a homotopy through singular contact structures}, we mean a continuous path in the space $\mf{S}$. In this section we show that any two contact structures $\xi_0, \xi_1 \in \Xi$ are connected by a path in $\mf{S}$. This means that homotopies through singular contact structures can be used to change the usual homotopy invariants of a contact structure.

To achieve this, we define a singular version of a well-known type of surgery on contact manifolds called the Lutz twist~\cite{geiges08}. In a tubular neighbourhood of a transverse curve $K$ in a contact structure $\xi$, there are coordinates $(r, \theta, \phi)$ such that $\xi$ is defined by $\eta_\text{std} = d\phi + r^2d\theta$. A Lutz twist excises this neighbourhood and glues in the contact structure defined by the 1-form $\eta_\text{Lutz} = h_1(r)d\phi + h_2(r)d\theta$, where $h_1, h_2$ are functions satisfying
\begin{enumerate}
  \item $h_1(r) = -1$ and $h_2(r) = -r^2$ near $r=0$,
  \item $h_1(r) = 1$ and $h_2(r) = r^2$ near $r=1$,
  \item $(h_1, h_2)$ is never parallel to $(h_1^\prime, h_2^\prime)$ when $r \neq 0$.
\end{enumerate}
The topological properties of this 1-form are determined by the first two conditions, while the third ensures it is contact. The result of this surgery is to make the contact structure overtwisted and to change the Euler class by $-2PD[K]$. Since the Euler class can change it is clear that this surgery cannot be effected by a homotopy through plane fields, however we will see that it can be effected via a homotopy in the larger space $\mf{S}$.

Our construction is based on the unfolding of the $A_2$ singularity~\cite{arnold88} (also called the fold catastrophe), given by the function germ $\phi_s = x^2 - y^2 + z(z^2 + s)$. This defines a plane field $d\phi = 0$ which has no singularities when $s>0$ and has two chiral Morse singularities when $s<0$, so allowing for the birthing of a pair of chiral singularities from a previously nonsingular plane field. Define $f_s(\phi) = s - \frac{1}{2}\cos z$ and consider the family of 1-forms
\begin{equation} \label{eq:eta_st}
  \eta_s = f_sdz + xdx-ydy + t(f_s (xdy -ydx) + 2xydz),
\end{equation}
on the solid torus $D^2 \times S^1$, for $t \neq 0$ a constant and $s \in [-1, 1]$. Depending on the choice of $s$ this 1-form is singular, with singular points at $x=y=0$ and $\cos(z) = 2s$. As $s$ decreases from $s=1$ to $s=-1$, an $A_2$ singularity of index zero is born at $z = 0$ when $s=1/2$, then splits apart into singularities of Morse index 1 and Morse index 2 which move in opposite directions around the circle $x^2+y^2=0$ until they annihilate again at $z = \pi$ when $s=-1/2$. For $s < -1/2$, there are again no singularities. We compute that
\begin{equation}
  \eta_s \wedge d\eta_s = 2t \left(f_s^2 + \left(1-\frac{1}{4}\sin z\right)x^2 + \left(1+\frac{1}{4}\sin z \right)y^2 \right)dx\wedge dy\wedge dz,
\end{equation}
so that $\eta_s$ defines a (singular, depending on the value of $s$) contact structure on the solid torus for any constant $t \neq 0$. The contact structure is positive or negative depending on the sign of $t$; in keeping with convention, fix $t > 0$.

We wish to show that $\eta_s$ can be extended to a singular contact form on $D^2 \times S^1$ that agrees with $\eta_\text{std}$ close to the boundary. We will use a result about characterstic foliations.
\begin{proposition} \label{prop:char_foliation_homotopy}
  Let $S$ be a surface embedded in $M$. If $\xi_0, \xi_1$ are contact structures inducing the same characteristic foliation on $S$, then there is a neighbourhood of $S$ in $M$ such that $\xi_0$ and $\xi_1$ are homotopic when restricted to this neighbourhood.
\end{proposition}
\noindent This implies that two contact structures inducing the same characteristic foliation on `opposite sides' of some surface can be glued together along that surface after some homotopy.

Take polar coordinates $r, \theta$ on $D^2$ and $z$ on $S^1$. We can view $D^2 \times S^1$ as a smaller copy of $D^2 \times S^1$ with a family of tori $T^2 \times [0,1]$ glued onto the outside. To construct our desired plane field, we will define a tight contact structure on $T^2 \times [0,1]$ such that the characteristic foliation induced on the surface $T^2 \times 0$ agrees with that induced on the boundary of the solid torus $D^2 \times S^1$ by $\eta_s$, and that is also homotopic to $\eta_\text{std}$ near $T^2 \times 1$. We can then glue using Proposition \ref{prop:char_foliation_homotopy} to produce a contact structure with the desired properties on $D^2 \times S^1$.

Consider a torus of radius $R$ inside $D^2 \times S^1$. Since the parameter $t$ is arbitrary, we can choose $t=1/R$. Then the characteristic foliation induced by $\eta_s$ on this torus is directed by the vector field $(f_s + \sin 2\theta)(\partial_\theta+\partial_z)$. Away from the line $0 \times S^1$ all the contact structures $\eta_s$ are homotopic, and therefore we may assume that on this torus $s$ is large enough so that $f_s + \sin 2\theta$ is never zero. Thus the characteristic foliation is directed by the vector field $\partial_\theta+\partial_z$. The characteristic foliation induced by $\eta_\text{std}$ on a torus of radius $R$ is $\partial_\theta - R^2 \partial_z$. Without loss of generality we can take $R=1$.

Now we define a tight contact structure on $T^2 \times [0,1]$ that interpolates between these characteristic foliations. There are many options, for example the contact structure defined by the 1-form $dz + \sqrt{2} \sin\left (\frac{\pi}{2}(r-\frac{1}{2}) \right)d\theta$, where $\theta, z$ denote the coordinates on $T^2$ and $r$ the coordinate on $[0,1]$. By glueing these pieces together we obtain a singular contact structure on $D^2 \times S^1$ that agrees with the standard neighbourhood of a transverse curve close to the boundary of the solid torus and has the properties of $\eta_s$ close to the core. Denote this singular contact structure by $\zeta_s$.

The contact structure $\zeta_1$ is homotopic through contact structures to the contact structure defined by $\eta_\text{std}$, which is tight and therefore satisfies the Bennequin inequalities, while the properties of $\zeta_{-1}$ are the same as those of the contact structure defined by $\eta_\text{Lutz}$. If $\xi^\prime$ is the plane field obtained from a contact structure $\xi$ by removing a neighbourhood of a transverse knot and glueing in $\zeta_{-1}$, then obviously $\xi$ and $\xi^\prime$ are plane fields which are homotopic in $\mf{S}$.
\begin{lemma} \label{lemma:change_in_e}
  If $\xi$ is a contact structure and $\xi^\prime$ is the plane field obtained by replacing a neighbourhood of some transverse curve $K$ in $\xi$ by $\zeta_{-1}$, then $e(\xi^\prime) = e(\xi) -2 PD[K]$ and $\xi^\prime$ is overtwisted.
\end{lemma}
\begin{proof}
  The argument is the same as for the Lutz twist, see~\cite[\S4.3]{geiges08}. We move from $s=1$ to $s=-1$ in $\zeta_s$, a pair of singularities is born which trace out (along the section of $K$ in between them) a set of overtwisted disks. Once the singularities have been annihilated, these disks persist. A vector field tangent to $\xi$ outside of the neighbourhood of $K$ extends to a vector field tangent to $\xi^\prime$ with an additional zero of winding $-1$ along a curve isotopic to $K$. Since the zero set of such a vector field generically determines the Euler class, we have $e(\xi^\prime) = e(\xi) -2 PD[K]$.
 \end{proof}

\begin{figure}[tb]
\centering
\includegraphics[width=1.0\linewidth]{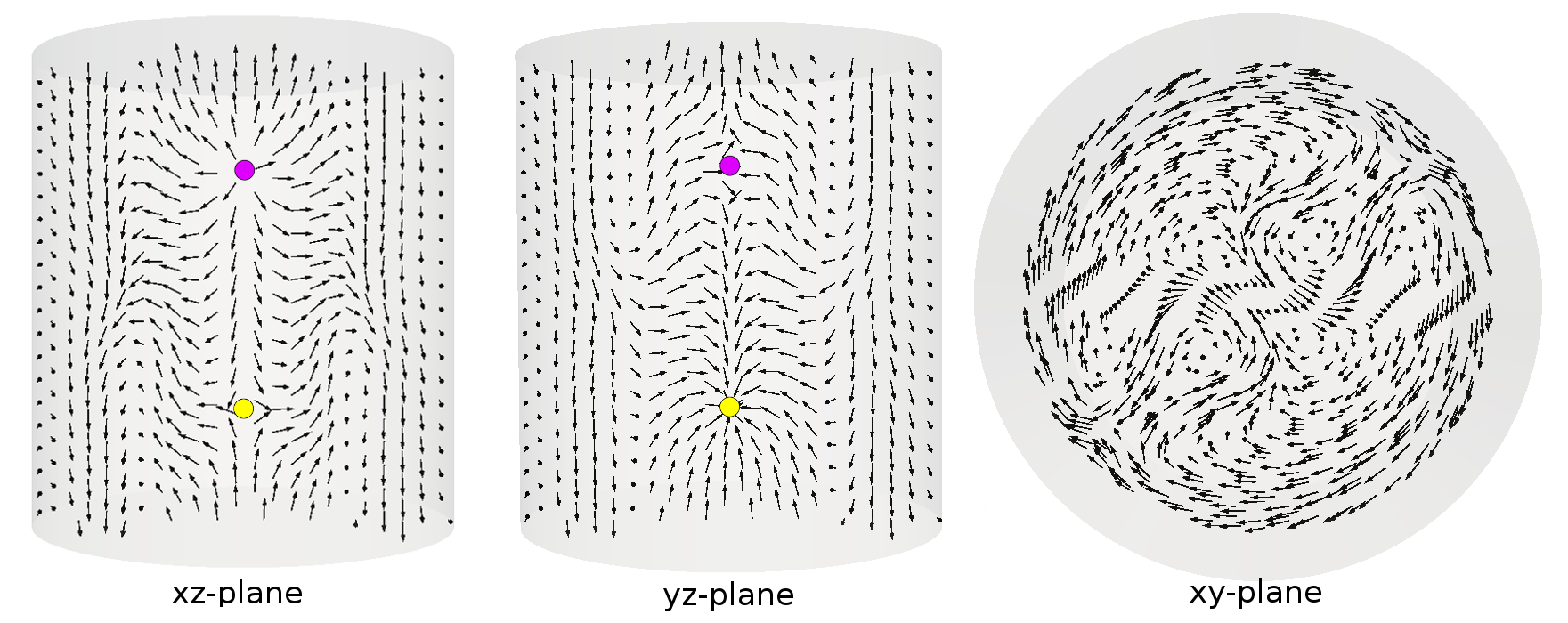}
\caption{The Reeb field of the contact form (\ref{eq:eta_st}) with $s=0$ as a vector field on $D^2 \times S^1$ with coordinates $x,y$ on the disk and $z$ on $S^1$, shown on various slices. The top and bottom of the cylinders shown are identified. The singularities are indicated by yellow (Morse index 2) and magenta (Morse index 1) disks. The Reeb field matches onto a boundary condition defined by the Reeb field of the model of a standard neighbourhood of a transverse curve in a contact structure.}
\label{fig:lutz}
\end{figure}

\begin{theorem} \label{thm:no_invariants_singular_plane_fields}
  Any pair of contact structures are connected by a path in $\mf{S}$.
\end{theorem}
\begin{proof}
  Let $\xi_0, \xi_1$ be two contact structures, which we can assume without loss of generality are overtwisted, since if either is tight, we can make them overtwisted using the singular Lutz twist. There is an overtwisted contact structure in every homotopy class of plane fields and, by Eliashberg~\cite{eliashberg89}, overtwisted contact structre are homotopic through contact structure if and only if they belong to the same homotopy class of plane fields. If $\xi_0$ does not belong to the same homotopy class as $\xi_1$, then we use singular Lutz twists to change its homotopy class using the same procedures that are used in the proof of the Lutz--Martinet Theorem, see~\cite{geiges08}. Any contact structure obtained from $\xi_0$ by applying singular Lutz twists is evidently homotopic to $\xi_0$ in the space $\mf{S}$.
 \end{proof}
\noindent Physically, the singular Lutz twist models the formation or destruction of tubes of overtwisted disks and the change in the homotopy invariants that occurs in a cholesteric liquid crystal material, where the evolution of the director traces out a continuous path in the space $\mf{S}$. Since the singularities we have used in this construction are Beltrami singularities, a similar statement can be made for the evolution of a time-dependent Beltrami field, and the formation of overtwisted disks and stagnation points that can occur.

\subsection{The Singular Weinstein Conjecture}
Notice that the Reeb field of $\eta_{-1}$ has a closed periodic orbit that is isotopic to the line $r=0$. From this we can deduce that periodic orbits can also be inserted into (or removed from) the Reeb field via a homotopy in $\mf{S}$.
\begin{proposition} \label{prop:adding_periodic_orbit}
  Any contact form $\eta_0$ can be homotoped through singular contact forms $\eta_t$ via a homotopy localised in a neighbourhood of some transverse curve $K$ so that the Reeb field of $\eta_1$ has a periodic orbit that is isotopic to $K$.
\end{proposition}
\noindent Note that the periodic orbits produced in this fashion have a saddle-like local structure and are therefore unstable.

The Weinstein conjecture (now proved in dimension 3~\cite{taubes07}) states that the Reeb field of any contact form on a closed manifold contains a periodic orbit. This conjecture is certainly not true for the Reeb-like fields of singular contact structures. Counterexamples are given in~\cite{miranda18} via a plug construction, which modifies the singular plane field in a neighbourhood of a point. We can use the singular Lutz twist to construct similar families of counterexamples. Take any contact form whose Reeb field has finitely many closed orbits. These orbits are closed transversal curves, so we use the singular Lutz twist to replace a neighbourhood of any closed orbit with the plane field $\zeta_0$, destroying the periodic orbit and replacing it with a pair of orbits that limit on the new singularities, that is, a homoclinic orbit. Repeating this process for each periodic orbit gives a singular contact structure whose Reeb-like fields have no periodic orbits.

The natural variant of the Weinstein conjecture for singular contact forms is as follows: every Reeb-like field of a singular contact form contains either a closed orbit, or an orbit that approaches the singular set in both forward and backward time~\cite{miranda18}. We refer to this as the singular Weinstein conjecture. Obviously any counterexamples to the Weinstein conjecture we produce using the singular Lutz twist satisfy the singular version of the conjecture, and the conjecture would be true if one could show that every singular contact structure could be produced from a contact structure by inserting and then unfolding index 0 singularities.

We conclude by giving two cases where the singular Weinstein conjecture holds.
\begin{proposition}
  The singular Weinstein conjecture holds for ABC fields.
\end{proposition}
\begin{proof}
  The space of ABC fields is described by two parameters $B, C \in [0,1]$. Since the parameter space is connected we can find a path connecting every singular ABC field to a nonsingular one. As we move along this path the singularities must annihilate with one another, eventually disappearing via unfoldings of index 0 singularities as they cross over the boundary $B^2+C^2=1$. This implies that there are orbits connecting the singularities which are removed when the singularities are brought together in this unfolding, and hence there are always either periodic orbits or trajectories connecting singularities for $B^2+C^2 \ge 1$. The resolution of the Weinstein conjecture in the contact case shows that periodic orbits exist whenever $B^2 + C^2 < 1$.
 \end{proof}
\begin{proposition}
  Let $M$ be a closed 3-manifold with $H_1(M) = 0$ and $\phi$ an intrinsically harmonic function. Then a singular contact form $\eta$ obtained by a small linear perturbation of $d\phi$ as in Proposition \ref{prop:global_perturbation} does not satisfy the Weinstein conjecture, but does satisfy the singular Weinstein conjecture.
\end{proposition}
\begin{proof}
  A gradient vector field $\nabla \phi$ does not contain periodic orbits, and since Reeb-like fields of $\eta$ are close to $\nabla \phi$ they will not do so either. However, $\nabla \phi$ does have orbits connecting the singularities, and therefore so do the Reeb-like fields of $\eta$.
 \end{proof}
A similar statement can be made if the singular contact form is produced by a perturbation of a closed 1-form using the method of Altschuler~\cite{altschuler95}, as discussed in \S\ref{sec:homotopy_invariants}.


\begin{thebibliography}{99}

  \bibitem{arnold89} V.I. Arnold, \textit{Mathematical Methods of Classical Mechanics, 2nd Ed.}, {Graduate Texts in Mathematics {\bf 60}, Springer--Verlag, (1989)}.
  \bibitem{arnold90} V.I. Arnold, Contact geometry and wave propagation, {Enseign. Math. (2) {\bf 36}, 56 (1990)}.
  \bibitem{arnold98} V.I. Arnold and B.A. Khesin, \textit{Topological Methods in Hydrodynamics}, {Applied Mathematical Sciences {\bf 125}, Springer, (1998)}.
  \bibitem{etnyre00} J. Etnyre and R. Ghrist, Contact topology and hydrodynamics: I. Beltrami fields and the Seifert conjecture, \href{https://doi.org/10.1088/0951-7715/13/2/306}{Nonlinearity {\bf 13}, 441 (2000)}.
  \bibitem{chandrasekhar57} S. Chandrasekhar and P.C. Kendall, On force-free magnetic fields, \href{https://doi.org/10.1086/146413}{Astrophys. J. {\bf 126}, 457 (1957)}
  \bibitem{chandrasekhar58} S. Chandrasekhar and L. Woltjer, On force-free magnetic fields, \href{https://doi.org/10.1073/pnas.44.4.285}{Proc. Natl. Acad. Sci. U.S.A. {\bf 44}, 285 (1958)}
  \bibitem{woltjer58} L. Woltjer, A theorem on force-free magnetic fields, \href{https://doi.org/10.1073/pnas.44.6.489}{Proc. Natl. Acad. Sci. U.S.A. {\bf 44}, 489 (1958)}
  \bibitem{ginzburg94} V. Ginzburg and B. Khesin, Steady fluid flows and symplectic geometry, \href{https://www.sciencedirect.com/science/article/pii/039304409490006X}{J. Geom. and Phys. {\bf 14}, 195 (1994)}.
  \bibitem{constantin88} P. Constantin and A. Majda, The Beltrami spectrum for incompressible fluid flows, \href{https://link.springer.com/article/10.1007/BF01218019}{Comm. Math.Phys. {\bf 115}, 435 (1988)}
  \bibitem{cardona19} R. Cardona, E. Miranda and D. Peralta-Salas, Euler flows and singular geometric structures, \href{https://arxiv.org/abs/1902.00039}{arXiv:1902.00039 {\bf [math.SG]} (2019)}.
  \bibitem{miranda18} E. Miranda and C. Oms, Contact structures with singularities, \href{https://arxiv.org/abs/1806.05638}{arXiv:1806.05638 {\bf [math.SG]} (2018)}.
  \bibitem{milde13} P. Milde \textit{et al.}, Unwinding of a skyrmion lattice by magnetic monopoles, \href{https://doi.org/10.1126/science.1234657}{Science {\bf 340}, 1076 (2013)}.
  \bibitem{eliashberg97} Y.M. Eliashberg and W.P. Thurston, \textit{Confoliations}, {University Lecture Series {\bf 13}, AMS, (1997)}.
  \bibitem{arnold12} V.I. Arnold, S.M Gusein--Zade, and A. N. Varchenko, \textit{Singularities of Differentiable Maps, Volume 1}, {Birkh\:auser (2012)}.
  \bibitem{ball2017} J.M. Ball, Mathematics and Liquid Crystals, \href{https://doi.org/10.1080/15421406.2017.1289425}{Mol. Cryst. Liq. Cryst. {\bf 647}, 1 (2017)}.
  \bibitem{deGennesProst} P.G. de Gennes and J. Prost, \textit{The Physics of Liquid Crystals, 2nd Ed.}, {Clarendon Press, Oxford, (1995)}.
  \bibitem{alexander12} G.P. Alexander, B.G. Chen, E.A. Matsumoto, and R.D. Kamien, Colloquium: Disclination loops, point defects, and all that in nematic liquid crystals, \href{https://doi.org/10.1103/RevModPhys.84.497}{Rev. Mod. Phys. {\bf 84}, 497 (2012)}.
  \bibitem{janich87} K. J\"{a}nich, Topological properties of ordinary nematics in 3-space, \href{https://doi.org/10.1007/BF00046687}{Acta App. Math. {\bf 8}, 65 (1987)}.
  \bibitem{mermin79} N.D. Mermin, The topological theory of defects in ordered media, \href{https://doi.org/10.1103/RevModPhys.51.591}{Rev. Mod. Phys. {\bf 51}, 591 (1979)}.
  \bibitem{monastyrsky86} M.I. Monastyrsky and V.S. Retakh, Topology of linked defects in condensed matter, \href{https://doi.org/10.1007/BF01211760}{Commun. Math. Phys. {\bf 103}, 445 (1986)}.
  \bibitem{chen09} B.G. Chen, G.P. Alexander, and R.D. Kamien, Symmetry breaking in smectics and surface models of their singularities, \href{https://doi.org/10.1073/pnas.0905242106}{Proc. Natl. Acad. Sci. U.S.A. {\bf 106}, 15577 (2009)}.
  \bibitem{machon19} T. Machon, H. Aharoni, Y. Hu, and R.D. Kamien, Aspects of defect topology in smectic liquid crystals, \href{https://doi.org/10.1007/s00220-019-03366-y}{Commun. Math. Phys. (2019)} https://doi.org/10.1007/s00220-019-03366-y.
  \bibitem{poenaru81} V. Po\'enaru, Some aspects of the theory of defects of ordered media and gauge fields related to foliations, \href{https://doi.org/10.1007/BF01213598}{Commun. Math. Phys. {\bf 80}, 127 (1981)}.
  \bibitem{posnjak17} G. Posnjak, S. \v{C}opar, and I. Mu\v{s}evi\v{c}, Hidden topological constellations and polyvalent charges in chiral nematic droplets, \href{http://doi.org/10.1038/ncomms14594}{Nat. Comm. {\bf 8}, 14594 (2017)}.
  \bibitem{machon17} T. Machon, Contact topology and the structure and dynamics of cholesterics, \href{https://doi.org/10.1088/1367-2630/aa958d}{New J. Phys. {\bf 19}, 113030 (2017)}.
  \bibitem{pollard19} J. Pollard, G. Posnjak, S. \v{C}opar, I. Mu\v{s}evi\v{c}, and G.P. Alexander, Point Defects, Topological Chirality, and Singularity Theory in Cholesteric Liquid-Crystal Droplets, \href{https://doi.org/10.1103/PhysRevX.9.021004}{Phys. Rev. X {\bf 9}, 021004 (2019)}.
  \bibitem{ackerman2016} P.J. Ackerman and I.I. Smalyukh, Reversal of helicoidal twist handedness near point defects of confined chiral liquid crystals, \href{https://doi.org/10.1103/PhysRevX.7.011006}{Phys. Rev. E {\bf 93}, 052702 (2016)}.
  \bibitem{arnold88} V.I. Arnold, V. Goryunov, O. Lyashko, and V. Vasil'ev, \textit{Singularity Theory}, {Springer, (1988)}.
  \bibitem{blair10} D.E. Blair, \textit{Riemannian Geometry of Contact and Sympletic Manifolds, 2nd Ed.}, {Birkha\"{u}ser, Progress in Mathematics {\bf 203}, (2010)}.
  \bibitem{farber03} M. Farber, \textit{The Topology of Closed One-Forms}, {Mathematical Surveys and Monographs {\bf 108}, AMS, (2003)}.
  \bibitem{etnyre12} J. Etnyre, R. Komendarczyk, and P. Massot, Tightness in contact metric 3-manifolds, \href{https://doi.org/10.1007/s00222-011-0355-2}{Invent. Math. {\bf 188}, 621 (2012)}.
  \bibitem{malgrange76} B. Malgrange, Frobenius avec singulariti\'{e}s. I. codimension un, \href{http://www.numdam.org/item/PMIHES_1976__46__163_0/}{Inst. Hautes \'{E}tudes Sci. Publ. Math. {\bf 46}, 163 (1976)}.
  \bibitem{calabi69} E. Calabi, An intrinsic characterisation of harmonic 1-forms, {Global Analysis, Papers in Honour of K. Kodaira, 101 (1969)}.
  \bibitem{arnold78} V.I. Arnold, Index of a singular point of a vector field, the Petrovski--Oleinik inequality, and mixed Hodge structures, \href{https://doi.org/10.1007/BF01077558}{Funct. Anal. Its Appl. {\bf 12}, 1 (1978)}.
  \bibitem{geiges08} H. Geiges, \textit{An Introduction to Contact Topology}, {Cambridge Studies in Advanced Mathematics {\bf 109}, Cambridge University Press, (2008)}.
  \bibitem{dombre86} T. Dombre, U. Frisch, J. Greene, M. H\'{e}non, A. Mehr, and A. Soward, Chaotic Streamlines in the ABC Flows, \href{https://doi.org/10.1017/S0022112086002859}{J. Fluid Mech. {\bf 167}, 353 (1986)}.
  \bibitem{novikov64} S.P. Novikov, Foliations of codimension 1 on manifolds, {Dokl. Akad. Nauk. SSSR {\bf 155}, 1010 (1964)}.
  \bibitem{moerdijk03} I. Moerdijk and J. Mr\v{c}un, \textit{Introduction to Foliations and Lie Groupoids}, {Cambridge Studies in Advanced Mathematics {\bf 91}, Cambridge University Press (2003)}.
  \bibitem{etnyre13} J. Etnyre, On knots in overtwisted contact structures, \href{https://doi.org/10.4171/QT/39}{Quantum Topol. {\bf 4}, 229 (2013)}.
  \bibitem{altschuler95} S. Altschuler, A geometric heat flow for one-forms on three-dimensional manifolds, \href{https://doi.org/10.1215/ijm/1255986631}{Illinois J. Math. {\bf 39}, 98 (1995)}.
  \bibitem{ellis18} P.W. Ellis, K. Nayani, J.P. McInerney, D.Z. Rocklin, J.O. Park, M. Srinivasarao, E.A. Matsumoto, and A. Fernandez-Nieves, Curvature Induced Twist in Homeotropic Nematic Tori, \href{https://doi.org/10.1103/PhysRevLett.121.247803}{Phys. Rev. Lett. {\bf 121}, 247803 (2018)}.
  \bibitem{musevic17} I. Mu\v{s}evi\v{c}, \href{https://doi.org/10.1007/978-3-319-54916-3_7}{Nematic Microdroplets, Shells and Handlebodies}, in {\sl Liquid Crystal Colloids}, Springer, Cham (2017).
  \bibitem{pairam13} E. Pairam, J. Vallamkondu, V. Koning, B.C. van Zuiden, P.W. Ellis, M.A. Bates, V. Vitelli, and A. Fernandez-Nieves, Stable nematic droplets with handles, \href{https://doi.org/10.1073/pnas.1221380110}{Proc. Natl. Acad. Sci. U.S.A. {\bf 110}, 9295 (2013)}.
  \bibitem{darmon16} A. Darmon, M. Benzaquen, S. \v{C}opar, O. Dauchota, and T. Lopez-Leon, Topological defects in cholesteric liquid crystal shells, \href{https://doi.org/10.1039/C6SM01748G}{Soft Matter {\bf 12}, 9280 (2016)}.
  \bibitem{eliashberg89} Y. Eliashberg, Classification of overtwisted contact structures on 3-manifolds, \href{https://doi.org/10.1007/BF01393840}{Invent. Math. {\bf 98}, 623 (1989)}.
  \bibitem{giroux01} E. Giroux, Structures de contact sur les vari\'{e}t\'{e}s fibr\'{e}es en cercles au-dessus d'une surface, \href{https://doi.org/10.1007/PL00000378}{Comment. Math. Helv. {\bf 76}, 218 (2001)}.
  \bibitem{taubes07} C.H. Taubes, The Seiberg-Witten equations and the Weinstein conjecture, \href{https://doi.org/10.2140/gt.2007.11.2117}{Geom. Topol. {\bf 11}, 2117 (2007)}.






\end{thebibliography}
\end{document}